\documentclass[11pt]{article}
\usepackage{cite}
\usepackage{mathrsfs} 
\usepackage{ifthen}
\usepackage{fancyhdr}
\usepackage[us,12hr]{datetime} 
\usepackage{exscale}
\usepackage{tabularx}
\usepackage{latexsym}
\usepackage{amsmath,amssymb,amstext,amsthm} 
\usepackage{graphicx,color,epsfig}
\usepackage[table]{xcolor}
\usepackage{graphicx}
\usepackage{fullpage}        
\usepackage{float}
\usepackage{verbatim}
\usepackage{enumerate}
\usepackage[bookmarks,pagebackref,
    pdfpagelabels=true, 
    ]{hyperref}
    \numberwithin{equation}{section}
\usepackage{makeidx}
\usepackage{tikz}
\usepackage[final]{pdfpages}
\usepackage{algorithm,algorithmic} 
\usepackage{lineno}
\numberwithin{table}{section}
\numberwithin{algorithm}{section}
\makeindex

\def\R{\mathbb{R}}
\def\Sc{\mathbb{S}}
\def\Sn{\Sc^n}

\def\Snp{\Sc_+^n}
\def\Srp{\Sc_+^r}
\def\Snrp{\Sc_+^{n-r}}
\def\Snpp{\Sc_{++}^n}
\def\Srpp{\Sc_{++}^r}
\def\Snrpp{\Sc_{++}^{n-r}}

\def\Rk{\mathbb{R}^k}

\def\Rn{\mathbb{R}^n}

\def\Rq{\mathbb{R}^q}

\def\eqref#1{{\normalfont(\ref{#1})}}

\def\eqref#1{{\normalfont(\ref{#1})}}

\newtheorem{theorem}{Theorem}[section]

\newtheorem{definition}[theorem]{Definition}
\newtheorem{example}[theorem]{Example}
\newtheorem{assump}[theorem]{Assumption}
\newtheorem{prop}[theorem]{Proposition}

\newtheorem{corollary}[theorem]{Corollary}

\newtheorem{remark}[theorem]{Remark}

\newtheorem{problem}[theorem]{Problem}

\newtheorem{lemma}[theorem]{Lemma}

\newcommand{\textdef}[1]{\textit{#1}\index{#1}}

\newcommand{\Xa}{{X(\alpha)}}
\newcommand{\ya}{{y(\alpha)}}
\newcommand{\Za}{{Z(\alpha)}}

\newcommand{\Xoa}{{X(\alpha)}}

\newcommand{\LL}{{\mathcal L} }

\newcommand{\DD}{{\mathcal D} }

\newcommand{\UU}{{\mathcal U} }
\newcommand{\VV}{{\mathcal V} }

\newcommand{\SDP}{\textbf{SDP}\,}
\newcommand{\SDPs}{\textbf{SDPs}\,}
\newcommand{\SDPsp}{\textbf{SDPs}}
\newcommand{\SDPR}{\textbf{SDP-R}\,}
\newcommand{\DSDP}{\textbf{D-SDP}\,}
\newcommand{\SDPp}{\textbf{SDP}}

\newcommand{\Rm}{{\R^m\,}}
\newcommand{\Rtn}{{\R^{\scriptsize{t(n)}}\,}}

\newcommand{\NN}{{\mathbb N}}

\newcommand{\A}{{\mathcal A}}

\newcommand{\F}{{\mathcal F\,}}
\newcommand{\N}{{\mathcal N\,}}

\newcommand{\bbm}{\begin{bmatrix}}
\newcommand{\ebm}{\end{bmatrix}}
\newcommand{\bem}{\begin{pmatrix}}
\newcommand{\eem}{\end{pmatrix}}
\newcommand{\beq}{\begin{equation}}
\newcommand{\beqs}{\begin{equation*}}
\newcommand{\bet}{\begin{table}}
\newcommand{\eeq}{\end{equation}}
\newcommand{\eeqs}{\end{equation*}}
\newcommand{\beqr}{\begin{eqnarray}}

\DeclareMathOperator{\face}{face}
\DeclareMathOperator{\sd}{sd}

\DeclareMathOperator{\nul}{null}
\DeclareMathOperator{\range}{range}

\DeclareMathOperator{\kvec}{{vec}}

\DeclareMathOperator{\trace}{{trace}}

\DeclareMathOperator{\diag}{{diag}}

\DeclareMathOperator{\svec}{{svec}}

\DeclareMathOperator{\sMat}{{sMat}}

\DeclareMathOperator{\relint}{{relint}}

\DeclareMathOperator{\cl}{{cl}}

\DeclareMathOperator{\rank}{{rank}}
\DeclareMathOperator{\spanl}{{span}}

\newcommand{\nc}{\newcommand}
\nc{\arrow}{{\rm arrow\,}}
\nc{\Arrow}{{\rm Arrow\,}}
\nc{\BoDiag}{{\rm B^0Diag\,}}
\nc{\bodiag}{{\rm b^0diag\,}}

\nc{\Mm}{{\mathcal M}^{m} }
\nc{\Mmn}{{\mathcal M}^{mn} }
\nc{\Mnr}{{\mathcal M}_{nr} }
\nc{\Mnmr}{{\mathcal M}_{(n-1)r} }
\nc{\kwqqp}{Q{$^2$}P\,}
\nc{\kwqqps}{Q{$^2$}Ps}

\nc{\notinaho}{(X,S)\in \overline{AHO}(\A)}
\nc{\inaho}{(X,S)\in AHO(\A)}

\newcommand{\bea}{\begin{eqnarray}}%
\newcommand{\eea}{\end{eqnarray}}%
\newcommand{\beas}{\begin{eqnarray*}}%
\newcommand{\eeas}{\end{eqnarray*}}%
\newcommand{\Rnn}{\R^{n \times n}}%

\renewcommand{\F}{\mathcal{F}}%
{}

\newcommand{\Hnp}[1][]{\,\mathbb{H}_+^{\ifthenelse{\equal{#1}{}}{n}{#1}}}
\newcommand{\Hn}[1][]{\,\mathbb{H}^{\ifthenelse{\equal{#1}{}}{n}{#1}}}
\newcommand{\Dn}[1][]{\,\mathbb{D}^{\ifthenelse{\equal{#1}{}}{n}{#1}}}

\newcommand{\ba}{b(\alpha) }
\newcommand{\Pa}{{\bf P(\alpha)}}

\newcommand{\Fa}{\F(\alpha)}

\begin{document}

\bibliographystyle{plain}
\title{
Complete Facial Reduction in One Step for Spectrahedra
}
             \author{
\href{https://uwaterloo.ca/combinatorics-and-optimization/about/people/ssremac}{Stefan
Sremac}\thanks{Department of Combinatorics and Optimization
        Faculty of Mathematics, University of Waterloo, Waterloo, Ontario, Canada N2L 3G1; Research supported by The Natural Sciences and Engineering Research Council of Canada and by AFOSR.
}
\and
\href{http://people.orie.cornell.edu/dd379}
{Hugo Woerdeman}\thanks{Department of Mathematics, Drexel University, 
3141 Chestnut Street, Philadelphia, PA 19104, USA.  Research supported by 
Simons Foundation grant 355645.
}
\and
\href{http://www.math.uwaterloo.ca/~hwolkowi/}
{Henry Wolkowicz}%
        \thanks{Department of Combinatorics and Optimization
        Faculty of Mathematics, University of Waterloo, Waterloo, Ontario, Canada N2L 3G1; Research supported by The Natural Sciences and Engineering Research Council of Canada and by AFOSR; 
\url{www.math.uwaterloo.ca/\~hwolkowi}.
}
}
\date{\today}
\fancypagestyle{plain}{
\fancyhf{}
\rfoot{Compiled on {\ddmmyyyydate\today} at \currenttime}
\lfoot{Page \thepage}
\renewcommand{\headrulewidth}{0pt}}
          \maketitle



\begin{abstract}
A spectrahedron is the feasible set of a semidefinite program, \SDPp,
i.e.,~the intersection of an affine set with the positive semidefinite cone.
While strict feasibility is a generic property for random problems, 
there are many classes of problems where strict feasibility fails and
this means that strong duality can fail as well. If the minimal
face containing the spectrahedron is known, the \SDP can easily be
transformed into an equivalent problem where strict feasibility holds
and thus strong duality follows as well.
The minimal face is fully characterized by the range or nullspace
of any of the matrices in its relative interior. Obtaining such a matrix
may require many \emph{facial reduction} steps and is currently not known to be a tractable problem
for spectrahedra with \emph{singularity degree} greater than one. We propose a \emph{single}
parametric optimization problem with a resulting type of \emph{central
path} and prove that the optimal solution 
is unique and in the relative interior 
of the spectrahedron. Numerical tests illustrate the efficacy of our
approach and its usefulness in regularizing \SDPsp.
\end{abstract}



\tableofcontents
\listoftables
\listoffigures

\section{Introduction}
\label{sec:intro}

A \textdef{spectrahedron} is the intersection of an affine manifold with the 
positive semidefinite cone.  Specifically, if \textdef{$\Sn$} denotes the set 
of $n\times n$ symmetric matrices, \textdef{$\Snp$}$ \subset \Sn$ denotes the set of positive semidefinite matrices, $\A:\Sn \rightarrow \Rm$ is a linear map, and $b \in \Rm$, then
\begin{equation}
\label{eq:feasset}
\textdef{$\F=\F(\A,b)$} := \{ X\in \Snp : \A(X) = b \}
\end{equation}
is a spectrahedron. We emphasize that $\F$ is given to us as a
function of the algebra, the data $\A,b$, rather than the geometry.

Our motivation for studying spectrahedra arises from 
\textdef{semidefinite programs, \SDPsp}\index{\SDP, semidefinite program}, where a linear objective is minimized 
over a spectrahedron.  In contrast to \textdef{linear programs}, strong duality is not an inherent property of \SDPs, but depends on a \textdef{constraint qualification (CQ)}\index{CQ, constraint qualification} such as the Slater CQ.  For an \SDP not satisfying the Slater CQ, the central path of the standard interior point algorithms is undefined and there is no guarantee of strong duality or convergence.  Although instances where the Slater CQ fails are pathological, see e.g. \cite{MR3622250} and \cite{Pataki2017}, they occur in many applications and this phenomenon has lead to the development of a number of regularization methods, \cite{RaTuWo:95,Ram:95,lusz00,int:deklerk7,LuoStZh:97}.

In this paper we focus on the \textdef{facial reduction} method, \cite{bw1,bw2,bw3}, where the optimization problem is restricted to the minimal face of $\Snp$ containing $\F$, denoted $\face(\F)$.  We note that the different regularization methods for \SDP are not fundamentally unrelated.  Indeed, in \cite{RaTuWo:95} a relationship between the extended dual of Ramana, \cite{Ram:95}, and the facial reduction approach is established and in\cite{MR3063940} the authors show that the dual expansion approach, \cite{lusz00,LuoStZh:97} is a kind of `dual' of facial reduction.  When knowledge of the minimal face is available, the optimization problem is easily transformed into one for which the Slater CQ holds.  Many of the applications of facial reduction to \SDP rely on obtaining the minimal face through analysis of the underlying structure.  See, for instance, the recent survey \cite{DrusWolk:16} for applications to hard combinatorial optimization and matrix completion problems.

In this paper we are interested in instances of \SDP where the minimal face can not be obtained analytically.  An algorithmic approach was initially presented in \cite{bw3} and subsequent analyses of this algorithm as well as improvements, applications to \SDP, and new approaches may be found in \cite{MR3108446,MR3063940,ScTuWonumeric:07,perm,permfribergandersen,2016arXiv160802090P,waki_mur_sparse}.  While these algorithms differ in some aspects, their main structure is the same.  At each iteration a subproblem is solved to obtain an \emph{exposing vector} for a face (not necessarily minimal) containing $\F$.  The \SDP is then reduced to this smaller face and the process repeated until the \SDP is reduced to $\face(\F)$.  Since at each iteration, the dimension of the ambient face is reduced by one, at most $n-1$ iterations are necessary.  We remark that this method is a kind of `dual' approach, in the sense that the exposing vector obtained in the subproblem is taken from the dual of the smallest face available at the current iteration.  We highlight two challenges with this approach: (1) each subproblem is itself an \SDP and thereby computationally intensive and (2) at each iteration a decision must be made regarding the rank of the exposing vector.

With regard to the first challenge, we note that it is really two-fold.  The computational expense arises from the complexity of an individual subproblem and also from the number of such problems to be solved.  The subproblems produced in \cite{ScTuWonumeric:07} are `nice' in the sense that strong duality holds, however, each subproblem is an \SDP and its computational complexity is comparable to that of the original problem.  In \cite{perm} a relaxation of the subproblem is presented that is less expensive computationally, but may require more subproblems to be solved.  The number of subproblems needed to solve depends of course on the structure of the problem but also on the method used to determine that facial reduction is needed.  For algorithms using the theorem of the alternative, \cite{bw1,bw2,bw3}, a theoretical lower bound, called the \emph{singularity degree}, is introduced in \cite{S98lmi}.  In \cite{MR2724357} an example is constructed for which the singularity degree coincides with the upper bound of $n-1$, i.e., the worst case exists.  In \cite{permfribergandersen}, the \textdef{self-dual embedding} algorithm of \cite{int:deklerk7} is used to determine whether facial reduction is needed.  This approach may require fewer subproblems than the singularity degree.    

The second challenge is to determine which eigenvalues of the exposing vector obtained at each iteration are identically zero, a classically challenging problem.  If the rank of the exposing vector is chosen too large, the problem may be restricted to a face which is smaller than the minimal face.  This error results in losing part of the original spectrahedron.  If on the other hand, the rank is chosen too small, the algorithm may require more iterations than the singularity degree.  The algorithm of \cite{ScTuWonumeric:07} is proved to be backwards stable only when the singularity degree is one, and the arguments can not be extended to higher singularity degree problems due to possible error in the decision regarding rank.  

Our main contribution in this paper is a `primal' approach to facial reduction, which does not rely on exposing vectors, but instead obtains a matrix in the relative interior of $\F$, denoted $\relint(\F)$
Since the minimal face is characterized by the range of any such matrix, we obtain a facially reduced problem in just one step.  As a result, we eliminate costly subproblems and require only one decision regarding rank.
\index{relative interior, $\relint(\cdot)$}
\index{$\relint(\cdot)$, relative interior}

While our motivation arises from \SDPsp, the problem of characterizing the relative interior of a spectrahedron is independent of this setting.  The problem is formally stated below.
\begin{problem}
\label{prob:main}
Given a spectrahedron $\F(\A,b) \subseteq \Sn$, find $\bar{X}\in \relint(\F)$.
\end{problem}

This paper is organized as follows.  In Section~\ref{sec:prelim} we introduce notation and discuss relevant material on \SDP strong duality and facial reduction.  We develop the theory for our approach in Section~\ref{sec:paramprob}, prove convergence to the relative interior, and prove convergence to the analytic center under a sufficient condition.  In Section~\ref{sec:projGN}, we propose an implementation of our approach and we present numerical results in Section~\ref{sec:numerics}.  We also present a method for generating instances of \SDP with varied singularity degree in Section~\ref{sec:numerics}.  We conclude the main part of the paper with an application to matrix completion problems in Section~\ref{sec:psdcyclecompl}.

\section{Notation and Background}
\label{sec:prelim}
Throughout this paper the ambient space is the Euclidean space of
$n\times n$ real symmetric matrices, $\Sn$, with the standard 
\textdef{trace inner product}
\[
\langle X,Y \rangle := \trace(XY) = \sum_{i=1}^n \sum_{j=1}^n X_{ij}Y_{ij},
\]
and the induced \textdef{Frobenius norm}
\[
\lVert X \rVert_F := \sqrt{\langle X, X\rangle }.
\]
In the subsequent paragraphs, we highlight some well known results on
the cone of positive semidefinite matrices and its faces, as well other
useful results from convex analysis.  For proofs and further reading we
suggest \cite{SaVaWo:97,MR2724357,con:70}.  The dimension of $\Sn$ is the triangular number $n(n+1)/2=: t(n)$.  
We define \textdef{$\svec$}$: \Sn \rightarrow \Rtn$ such that it maps the upper triangular 
elements of $X \in \Sn$ to a vector in $\Rtn$ where the off-diagonal 
elements are multiplied by $\sqrt{2}$.  Then $\svec$ is an isometry and an 
isomorphism with \textdef{$\sMat$}$ := \svec^{-1}$.  Moreover, for 
$X,Y \in \Sn$, 
\index{$t(n)$, triangular number}
\index{triangular number, $t(n)$}
\[
\langle X,Y \rangle = \svec(X)^T \svec(Y).
\]
The eigenvalues of any $X \in \Sn$ are real and indexed so as to satisfy,
\index{operator norm, $\|X\|$}
\index{$\|X\|$, operator norm}
\[
\lambda_1(X) \ge \lambda_2(X) \ge \cdots \ge \lambda_n(X),
\]
and $\lambda(X) \in \Rn$ is the vector consisting of all the eigenvalues.  
In terms of this notation, the operator 2-norm for matrices is defined as 
$\lVert X \rVert_2 := \max_i \lvert \lambda_i(X) \rvert$.  When the argument to $\| \cdot \|_2$ is a vector, this denotes the usual Euclidean norm.  
The Frobenius norm may also be expressed in terms of eigenvalues: $\lVert X \rVert_F=
\lVert \lambda(X) \rVert_2$.  The set of \textdef{positive semidefinite
(PSD)}\index{PSD, positive semidefinite matrices} matrices, $\Snp$, is a closed convex cone in $\Sn$, whose
interior consists of the \textdef{positive definite (PD)}\index{PD, positive definite matrices} matrices,
\textdef{$\Snpp$}.  The cone $\Snp$ induces the \textdef{L\"owner partial order} on $\Sn$.
That is, for $X,Y \in \Sn$ we write $X\succeq Y$ when $X-Y \in \Snp$ and similarly $X\succ Y$ when $X-Y \in \Snpp$.  For $X,Y \in \Snp$ the following equivalence holds:
\begin{equation}
\label{eq:innerprodmatrixprod}
\langle X, Y \rangle =0 \ \iff \ XY = 0.
\end{equation}
\begin{definition}[face]
\label{def:face}
A closed convex cone $f \subseteq \Snp$ is a \textdef{face} of $\Snp$ if 
\[
X,Y \in \Snp, \ X+Y \in f \ \implies \ X,Y \in f.
\]
\end{definition}
A nonempty face $f$ is said to be \emph{proper} if $f \ne \Snp$ and $f
\ne 0$.  Given a convex set $C \subseteq \Snp$, the \textdef{minimal face}\index{$\face(\cdot)$, minimal face}
of $\Snp$ containing $f$, with respect to set inclusion, is denoted
$\face(C)$.  A face $f$ is said to be \emph{exposed} if there exists $W \in \Snp \setminus \{0\}$ such that 
\[
f = \{X \in \Snp : \langle W, X \rangle = 0\}.
\]
Every face of $\Snp$ is exposed and the vector $W$ is referred to as an
\textdef{exposing vector}.  The faces of $\Snp$ may be characterized in terms of the range of any of its maximal rank elements.  Moreover, each face is isomorphic to a smaller dimensional positive semidefinite cone, as is seen in the subsequent theorem.
\begin{theorem}[\cite{DrusWolk:16}]
\label{thm:face}
Let $f$ be a face of $\Snp$ and $X \in f$ a maximal rank element with rank $r$ and orthogonal spectral decomposition
\[
X=\begin{bmatrix} V & U  \end{bmatrix}
\begin{bmatrix} D & 0 \cr 0 & 0 \end{bmatrix}
\begin{bmatrix} V & U  \end{bmatrix}^T \in \Snp, \quad D\in \Srpp.
\]
Then $f = V \Srp V^T$ and $\relint(f) = V \Srpp V^T$.  Moreover, $W \in \Snp$ is an exposing vector for $f$ if
and only if $W \in U\Snrpp U^T$.
\end{theorem}
We refer to $U\Snrp U^T$, from the above theorem, as the \textdef{conjugate
face}\index{$f^c$, conjugate face}, denoted $f^c$.  For any convex set $C$, an explicit form for $\face(C)$ and $\face(C)^c$ may be obtained from the orthogonal spectral decomposition of any of its maximal rank elements as in Theorem~\ref{thm:face}.

For a linear map $\A : \Sn \rightarrow \Rm$, there exist $S_1,
\dotso,S_m \in \Sn$ such that
\index{$(\A(X))_i = \langle X,S_i \rangle$}
\[
\begin{pmatrix} \A(X)\end{pmatrix}_i = \langle X,S_i \rangle, \quad \forall i \in \{1,\dotso,m\}. 
\]
The \textdef{adjoint} of $\A$ is the unique linear map $\A^* : \Rm \rightarrow \Sn$ satisfying 
\[
\langle \A(X),y \rangle = \langle X,\A^*(y) \rangle,
\quad \forall  X \in \Sn, \, y \in \Rm,
\]
and has the explicit form $\A^*(y) = \sum_{i=1}^m y_i S_i$,
i.e.,~$\range(\A^*)=\spanl \{S_1,\ldots,S_m\}$.
We define $A_i\in \Sn$ to form a basis for the nullspace,
$\nul(\A)=\spanl \{ A_1,\dotso,A_q\}$.
\index{$\range (\A^*)=\spanl \{ S_1,\dotso,S_m\}$}
\index{$\nul(\A)=\spanl \{ A_1,\dotso,A_q\}$}

For a non-empty convex set $C \subseteq \Sn$ the \textdef{recession cone}\index{$C^{\infty}$, recession cone}, denoted $C^{\infty}$, captures the directions in which $C$ is unbounded.  That is
\begin{equation}
\label{eq:recession}
C^{\infty} := \{Y \in \Sn : X + \lambda Y \in C, \ \forall \lambda \ge 0, \ X \in C \}.
\end{equation}
Note that the recession directions are the same at all points $X \in C$.  For a non-empty set $S \subseteq \Sn$, the \textdef{dual cone}\index{$S^+$, dual cone} (also referred to as the positive polar) is defined as
\begin{equation}
\label{eq:dualcone}
S^+ := \{ Y \in \Sn : \langle X, Y \rangle \ge 0, \ \forall X \in S\}.
\end{equation} 
A useful result regarding dual cones is that for cones $K_1$ and $K_2$,
\begin{equation}
\label{eq:dualintersection}
(K_1 \cap K_2)^+ = \cl(K_1^+ + K_2^+),
\end{equation}
where \textdef{$\cl(\cdot)$}\index{closure, $\cl(\cdot)$} denotes set closure.

\subsection{Strong Duality in Semidefinite Programming and Facial Reduction}
\label{sec:sdpstrongduality}
Consider the standard primal form SDP
\begin{equation}
\label{prob:sdpprimal}
\SDP \qquad \qquad \textdef{$p^{\star}$}:=\min \{ \langle C,X\rangle : \A(X)=b, X\succeq 0\},
\end{equation}
with Lagrangian dual
\begin{equation}
\label{prob:sdpdual}
\DSDP \qquad \qquad \textdef{$d^{\star}$}:=\min \{ b^Ty  : \A^*(y) \preceq C \}.
\end{equation}
Let $\F$ denote the spectrahadron defined by the feasible set of $\SDP$.
One of the challenges in semidefinite programming is that strong duality
is not an inherent property, but depends on a constraint qualification,
such as the Slater CQ.
\begin{theorem}[strong duality,~\cite{SaVaWo:97}]
\label{thm:strongduality}
If the primal optimal value $p^{\star}$ is finite and
$\F \cap \Snpp \ne \emptyset$, then the primal-dual pair $\SDP$ and
$\DSDP$ have a \textdef{zero duality gap}, $p^{\star}=d^{\star}$, and $d^{\star}$ is attained.
\end{theorem}
Since the Lagrangian dual of the dual is the primal, this result can
similarly be applied to the dual problem, i.e.,~if the primal-dual pair
both satisfy the Slater CQ, then there is a zero duality gap and both
optimal values are attained. 

Not only can strong duality fail with the absence of
the Slater CQ, but the standard central path of an interior point
algorithm is undefined.  The facial reduction regularization approach of \cite{bw1,bw2,bw3} restricts \SDP to the minimal face of $\Snp$ containing $\F$:
\begin{equation}
\label{eq:sdpr}
\SDPR \qquad \qquad \min \{\langle C,X \rangle : \A(X) = b,\, X \in \face(\F) \}.
\end{equation}  
Since the dimension of $\F$ and $\face(\F)$ is the same, the Slater CQ holds for the facially reduced problem.  Moreover, $\face(\F)$ is isomorphic to a smaller dimensional positive semidefinite cone, thus $\SDPR$ is itself a semidefinite program.  The restriction to $\face(\F)$ may be obtained as in the results of 
Theorem~\ref{thm:face}.  The dual of \SDPR restricts the slack variable 
to the dual cone 
\[
Z=C-\A^*(y)\in \face(\F)^+.
\]
Note that $\F^+=\face(\F)^+$.
If we have knowledge of $\face(\F)$, i.e., we have the matrix $V$ such that $\face(\F) = V\Srp V^T$, then we may 
replace $X$ in \SDP by $VRV^T$ with $R \succeq 0$.  After rearranging, we obtain \SDPR.  Alternatively, 
if our knowledge of the minimal face is in the form of an exposing vector, 
say $W$, then we may obtain $V$ so that its columns form a basis for $\nul(W)$.  
We see that the approach is straightforward when knowledge of $\face(\F)$ is available.  In instances where such knowledge is unavailable, the following theorem of the alternative from \cite{bw3} guarantees the 
existence of exposing vectors that lie in $\range(\A^*)$.
\begin{theorem}[of the alternative,~\cite{bw3}]
\label{thm:alternative}
Exactly one of the following systems is consistent:
\begin{enumerate}
\item $\A(X) = b$, $X\succ 0$,
\item $0 \ne \A^*(y) \succeq 0$, $b^Ty = 0$.
\end{enumerate}
\end{theorem}
The first alternative is just the Slater CQ, while if the second
alternative holds, then $\A^*(y)$ is an exposing vector for a face
containing $\F$.  We may use a basis for $\nul(\A^*(y))$ to obtain a smaller \SDP.  If the Slater CQ holds for the new \SDP we have obtained \SDPR, otherwise, we find an exposing vector and reduce the problem again.  We outline the facial reduction procedure in Algorithm~\ref{algo:fr}.  At each iteration, the dimension of the problem is reduced by at least one, hence this approach is bound to obtain \SDPR in at most $n-1$ iterations, assuming that the initial problem is feasible.  If at each iteration the
exposing vector obtained is of maximal rank then the number of
iterations required to obtain \SDPR is referred to as the \emph{singularity 
degree}, \cite{S98lmi}.  For a non-empty spectrahedron, $\F$, we denote the singularity
degree as $\sd=\sd(\F)$.  
\index{singularity degree, $\sd=\sd(\F)$}
\index{$\sd=\sd(\F)$, singularity degree}
\begin{algorithm}
\label{algo:fr}
\caption{Facial reduction procedure using the theorem of the alternative.}
\begin{algorithmic}
\STATE Initialize $S_i$ so that $(\A(X))_i = \langle S_i,X \rangle$ for $i \in \{1,\dotso,m\}$
\WHILE{Item 2. of Theorem~\ref{thm:alternative}}
\STATE {obtain exposing vector $W$}
\STATE {$W = \begin{bmatrix}
U & V \end{bmatrix}\begin{bmatrix}
D & 0 \\
0 & 0 \end{bmatrix}\begin{bmatrix}
U & V\end{bmatrix} , \quad D \succ 0$}
\STATE {$S_i \leftarrow V^TS_iV, \quad i \in \{1,\dotso, m\}$}
\ENDWHILE
\end{algorithmic}
\end{algorithm}

We remark that any algorithm pursuing the minimal face through exposing vectors of the form $\A^*(\cdot)$, must perform at least as many iterations as the singularity degree.  The singularity degree could be as large as the trivial upper bound $n-1$ as is seen in the example of \cite{MR2724357}.  Thus facial reduction may be very expensive computationally.  On the other hand, from Theorem~\ref{thm:face} we see that $\face(\F)$ is fully characterized by the range of any of its relative interior matrices.  That is, from any solution to Problem~\ref{prob:main} we may obtain the regularized problem \SDPR.

\section{A Parametric Optimization Approach}
\label{sec:paramprob}
In this section we present a parametric optimization problem that solves Problem~\ref{prob:main}. 
\begin{assump}
\label{assump:main}
We make the following assumptions:
\begin{enumerate}
\item $\A$ is surjective,
\item $\F$ is non-empty, bounded and contained in a proper face of $\Snp$.
\end{enumerate}
\end{assump}
The assumption on $\A$ is a standard regularity assumption and so is the
non-emptiness assumption on $\F$.  The necessity of $\F$ to be bounded
will become apparent throughout this section, however, our approach may
be applied to unbounded spectrahedra as well.  We discuss such extensions in Section~\ref{sec:unbounded}.  The assumption that $\F$ is contained in a proper face of $\Snp$ restricts our discussion to those instances of \SDP that are interesting with respect to facial reduction.

In the 
following lemma are stated two useful characterizations of 
bounded spectrahedra.
\begin{lemma}
\label{lem:boundedchar}
The following holds:
\[
\F \text{ is bounded} \ \iff \ \nul(\A) \cap \Snp= \{0\} \ \iff \range(\A^*) \cap \Snpp \ne \emptyset.
\]
\end{lemma}
\begin{proof}
For the first equivalence, $\F$ is bounded if and only if $\F^{\infty} =
	\{0\}$ by Theorem~8.4 of \cite{con:70}.  It suffices, therefore,
	to show that $\F^{\infty} = \nul(\A) \cap \Snp$.  It is easy to
	see that $(\Snp)^{\infty} = \Snp$ and that the recession cone of
	the affine manifold defined by $\A$ and $b$ is $\nul(\A)$.  By
	Corollary~8.3.3 of \cite{con:70} the recession cone of the intersection of convex sets is the intersection of the respective recession cones, yielding the desired result.

Now let us consider the second equivalence.  For the forward direction, observe that
\begin{align*}
\nul(\A) \cap \Snp = \{0\} \ &\iff \ \left( \nul(\A) \cap \Snp \right)^+ = \{0 \}^+, \\
& \iff \ \nul(\A)^{\perp} + \Snp = \Sn, \\
& \iff \ \range(\A^*) + \Snp = \Sn.
\end{align*}
The second inequality is due to \eqref{eq:dualintersection} and one can verify that in this case $\nul(\A)^{\perp} \cap \Snp$ is closed.  Thus there exists $X \in \range(\A^*)$ and $Y \in \Snp$ such that $X+Y=-I$.  Equivalently, $-X = I + Y \in \Snpp$.  For the converse, let $X \in \range(\A^*) \cap \Snpp$ and suppose $0\ne S \in \nul(\A) \cap \Snp$.  Then $\langle X,S \rangle = 0$ which implies, by \eqref{eq:innerprodmatrixprod}, that $XS = 0$.  But then $\nul(X) \ne \{0\}$, a contradiction.
\end{proof}
Let $r$ denote the maximal rank of any matrix in $\relint (\F)$ and let the columns of $V \in \R^{n\times r}$ form a basis for its range.  In seeking a relative interior point of $\F$ we define a specific point from which we develop a parametric optimization problem.   
\index{analytic center of $\F$, $\hat{X}$}
\index{$\hat{X}$, analytic center of $\F$}
\begin{definition}[analytic center]
\label{def:analytic}
The analytic center of $\F$ is the unique matrix $\hat{X}$ satisfying
\begin{equation}
\label{eq:analytic}
\hat{X} = \arg \max \{ \log \det (V^TXV) : X \in \F \}.
\end{equation}
\end{definition}
Under Assumption~\ref{assump:main} the analytic center is well-defined
and this follows from the proof of Theorem~\ref{thm:maxdet}, below.  It is easy to see that the analytic center is indeed in the
relative interior of $\F$ and therefore a solution to Probelm~\ref{prob:main}.  However, the optimization problem from which it is derived is 
intractable due to the unknown matrix $V$.  If $V$ is simply removed from the optimization problem (replaced with the identity), then the problem is ill-posed since the objective does not take any 
finite values over the feasible set as it lies on the 
boundary of the \SDP cone.  To combat these issues, we propose replacing $V$ with $I$ and also perturbing $\F$ so that it intersects $\Snpp$.  The perturbation we choose is that of replacing $b$ with 
$\ba := b+ \alpha \A(I), \ \alpha>0$, thereby defining a family of spectrahedra
\[
\Fa := \{X \in \Snp : \A(X) = \ba \}.
\]
It is easy to see that if $\F \ne \emptyset$ then $\Fa$ has postive definite elements for every $\alpha >0$.  Indeed $\F + \alpha I \subset \Fa$.  Note that the affine manifold may be perturbed by any positive definite matrix and $I$ is chosen for simplicity.  We now consider the family of optimization problems
for $\alpha > 0$:
\begin{equation}
\label{eq:Palpha}
\Pa  \qquad \qquad \max \{ \log \det ( X) : X\in \Fa \}.
\end{equation}
It is well known that the solution to this problem exists and is unique for each $\alpha > 0$.  We include a proof in Theorem~\ref{thm:maxdet}, below.  Moreover, since $\face(\Fa) = \Snp$ for each $\alpha > 0$, the solution to $\Pa$ is in $\relint(\Fa)$ and is exactly the analytic center of $\Fa$.  The intuition behind our approach is that as the perturbation gets smaller, i.e., $\alpha \searrow 0$, the solution to $\Pa$ approaches the relative interior of $\F$.  This intuition is validated in Section~\ref{sec:convergence}.  Specifically, we show that the solutions to $\Pa$ form a smooth path that converges to $\bar{X} \in \relint(\F)$.  We also provide a sufficient condition for the limit point to be $\hat{X}$ in Section~\ref{sec:analyticcenter}.

We note that our approach of perturbing the spectrahedron in order to use the $\log \det(\cdot)$ function is not entirely new.  In \cite{fazelhindiboyd:01}, for instance, the authors perturb a convex feasible set in order to approximate the rank function using $\log \det(\cdot)$.  Unlike our approach, their perturbation is constant.

\subsection{Optimality Conditions}
We choose the strictly concave function $\log \det (\cdot)$ for its 
elegant optimality conditions, though the maximization is equivalent 
to maximizing only the determinant. We treat it as an \textdef{extended
valued} concave function that takes the value $-\infty$ if $X$ is singular.
For this reason we refer to both functions $\det(\cdot)$ and $\log \det (\cdot)$ equivalently
throughout our discussion.

Let us now consider the optimality conditions for the problem $\Pa$.
Similar problems have been thoroughly studied throughout the literature
in matrix completions and \SDPp, 
e.g.,~\cite{GrJoSaWo:84,MR2807419,SaVaWo:97,MR1614078}.  Nonetheless, we include a proof for 
completeness and to emphasize its simplicity.
\begin{theorem}[optimality conditions]
\label{thm:maxdet}
For every $\alpha >0$ there exists a
unique $\Xa \in \Fa \cap \Snpp$ such that 
\begin{equation}
\label{eq:maxlogdet}
\Xa =\arg \max \{ \log \det (X) : X \in \Fa \}.
\end{equation}
Moreover, $\Xa$ satisfies \eqref{eq:maxlogdet} if, and only if, there exists a unique $\ya \in \Rm$ and a unique $\Za \in \Snpp$ such that
\begin{equation}
\label{eq:optimalsystem}
\begin{bmatrix}
 \A^*(\ya)-\Za  \\
\A(\Xa) - \ba \\
\Za \Xa - I
\end{bmatrix} = 0.
\end{equation}
\end{theorem}
\begin{proof}
By Assumption~\ref{assump:main}, $\F \ne \emptyset$ and bounded and it 
follows that $\Fa \cap \Snpp \ne \emptyset$ and by Lemma~\ref{lem:boundedchar} 
it is bounded.  Moreover, $\log \det (\cdot)$ is a strictly concave function
over $\Fa \cap \Snpp$ (a so-called barrier function) and
\[
\lim_{\det(X)\to 0} \log \det (X) = -\infty.
\]
Thus, we conclude that the optimum $\Xa \in \Fa \cap \Snpp$ exists and is 
unique. The Lagrangian of problem \eqref{eq:maxlogdet} is 
\begin{align*}
\LL(X,y) &= \log \det(X) - \langle y, \A(X) - b\rangle \\
&= \log \det(X) - \langle \A^*(y), X \rangle + \langle y, b\rangle.
\end{align*}
Since the constraints are linear,
stationarity of the Lagrangian holds at $\Xa$.  Hence there exists $\ya \in \Rm$ 
such that $(\Xa)^{-1} = \A^*(\ya) =: \Za$.  Clearly $\Za$ is unique, and
since $\A$ is surjective, we conclude in addition that $\ya$ is unique.
\end{proof}

\subsection{The Unbounded Case}
\label{sec:unbounded}
Before we continue with the convergence results, we briefly address 
the case of unbounded spectrahedra.
The restriction to bounded spectrahedra is necessary in order to have
solutions to \eqref{eq:maxlogdet}.  There are certainly large families
of \SDPs where the assumption holds.  Problems arising from liftings of
combinatorial optimization problems often have the diagonal elements
specified, and hence bound the corresponding spectrahedron.  Matrix
completion problems are another family where the diagonal is often
specified.  Nonetheless, many \SDPs have unbounded feasible sets and we provide two methods for reducing such spectrahedra to bounded ones.  First, we show that the boundedness of $\F$ may be determined by solving a projection problem.
\begin{prop}
\label{prop:boundtest}
Let $\F$ be a spectrahedron defined by the affine manifold $\A(X) = b$ and let 
\[
P := \arg \min \ \{\lVert X - I \rVert_F : X\in \range(\A^*) \}.
\]
Then $\F$ is bounded if $P \succ 0$.
\end{prop}
\begin{proof}
First we note that $P$ is well defined and a singleton since it is the projection of $I$ onto a closed convex set.  Now $P\succ 0$ implies that $\range(\A^*) \cap \Snpp \ne
\emptyset$ and by Lemma~\ref{lem:boundedchar} this is equivalent to $\F$
bounded.
\end{proof}
The proposition gives us a sufficient condition for $\F$ to be bounded.  Suppose this condition is not satisfied, but we have knowledge of some matrix $S \in \F$.  Then for $t > 0$, consider the spectrahedron
\[
\F' := \{ X \in \Sn : X\in \F, \ \trace(X) = \trace(S) + t \}.
\]
Clearly $\F'$ is bounded.  Moreover, we see that $\F' \subset \F$ and contains maximal rank elements of $\F$, hence $\face(\F') = \face(\F)$.  It follows that $\relint(\F') \subset \relint(\F)$ and we have reduced the problem to the bounded case.

Now suppose that the sufficient condition of the proposition does not hold and we do not have knowledge of a feasible element of $
F$.  In this case we detect recession
directions, elements of $\nul (\A) \cap \Snp$, and project to the orthogonal complement.  Specifically, if $\F$ is unbounded then $\Fa$ is unbounded and problem \eqref{eq:Palpha} is unbounded.  Suppose, we have detected unboundedness, i.e., we have $X \in \F(\alpha)\cap \Snp$ with large norm.  Then $X = S_0 + S$ with $S \in \nul (\A) \cap \Snp$ and $\lVert S \rVert \gg \lVert S_0 \rVert$.  We then restrict $\F$ to the orthogonal complement of $S$, that is, we consider the new spectrahedron
\[
\F' := \{X\in \Sn : X\in \F, \ \langle S,X\rangle = 0\}.
\]
By repeated application, we eliminate a basis for the recession directions and obtain a bounded spectrahedron.  From any of the relative interior points of this spectrahedron, we may obtain a relative interior point for $\F$ by adding to it the recession directions obtained throughout the reduction process.

\subsection{Convergence to the Relative Interior and Smoothness}
\label{sec:convergence}

By simple inspection it is easy to see that $(\Xa,\ya,\Za)$, as in \eqref{eq:optimalsystem}, does not converge as $\alpha \searrow 0$.  Indeed, under Assumption~\ref{assump:main},
\[
\lim_{\alpha \searrow 0} \lambda_n(\Xa) \rightarrow 0 \ \implies \ \lim_{\alpha \searrow 0} \lVert \Za \rVert_2 \rightarrow +\infty.
\]
It is therefore necessary to scale $\Za$ so that it remains bounded.  Let us look at an example.

\begin{example} Consider the matrix completion problem: find $X \succeq 0$ having the form 
$$ \begin{pmatrix} 1 & 1 &  ? \cr
 1 & 1 & 1  \cr ? & 1 & 1 \end{pmatrix}. $$
The set of solutions is indeed a spectrahedron with ${\mathcal A}$ and $b$ given by
$$ {\mathcal A} \left( \begin{bmatrix} x_{11} & x_{12} & x_{13} \cr
 x_{12} & x_{22}  &  x_{23}  \cr x_{13}  &  x_{23} &  x_{33} \end{bmatrix} \right) := \begin{pmatrix} x_{11} \cr x_{12} \cr x_{22} \cr x_{23} \cr x_{33} \end{pmatrix},\   b := \begin{pmatrix} 1 \cr
 1 \cr 1 \cr 1 \cr 1\end{pmatrix}. $$
In this case, it is not difficult to obtain
$$ X(\alpha ) = \begin{pmatrix}1+\alpha & 1 & \frac{1}{1+\alpha} \cr 1 & 1+\alpha & 1\cr \frac{1}{1+\alpha} &  1 & 1+\alpha \end{pmatrix},$$ with inverse
$$  X(\alpha)^{-1} =\frac{1}{\alpha(2+\alpha)} \begin{pmatrix}1+\alpha & -1 & 0 \cr -1 & \frac{\alpha^2+2\alpha+2}{ 1+\alpha} & -1\cr 0 & -1 & 1+\alpha \end{pmatrix}. $$
Clearly $ \lim_{\alpha \searrow 0} \lVert  X(\alpha)^{-1}  \rVert_2 \rightarrow +\infty.$
However, when we consider $\alpha X(\alpha)^{-1}$, and take the limit as $\alpha$ goes to 0 we obtain the
bounded limit 
$$ \bar{Z} = \begin{pmatrix} \frac12 & -  \frac12  & 0 \cr  - \frac12 & 1 & - \frac12 \cr  0 &- \frac12 &  \frac12 \end{pmatrix}. $$
Note that $\bar{X}= X(0)$ is the $3\times 3$ matrix with all ones, ${\rm rank} \bar{X}+ {\rm rank} \bar{Z}= 3$, and $\bar{X} \bar{Z} = 0$.
\end{example} 

It turns out that multiplying  $\Xa^{-1}$ by $\alpha$ always bounds the sequence $(\Xa,\ya,\Za)$.  Therefore, we consider the scaled system
\begin{equation}
\label{eq:scaledoptimality}
\begin{bmatrix}
\A^*(y) - Z \\
\A(X) - \ba \\
ZX - \alpha I
\end{bmatrix} = 0, \ X \succ 0, \ Z \succ 0, \ \alpha > 0,
\end{equation}
that is obtained from \eqref{eq:optimalsystem} by multiplying the last equation by $\alpha$. Abusing our previous notation, we let $(\Xa,\ya,\Za)$ denote a solution 
to \emph{this} system and we refer to the set of all such solutions as the
\textdef{parametric path}.  The parametric path has clear parallels to the \emph{central path} of \SDP, however, it differs in one main respect: it is not contained in the relative interior of $\F$.  In the main theorems of this section we prove that the parametric path is smooth and converges as $\alpha\searrow 0$ with the primal limit point in $\relint(\F)$.  We begin by showing that the primal component of the parametric path has cluster points.
\begin{lemma}
\label{lem:primalconverge}
Let $\bar{\alpha}> 0$.  For every sequence $\{\alpha_k\}_{k\in \NN} \subset (0,\bar{\alpha}]$ such that $\alpha_k \searrow 0$, there exists a subsequence $\{\alpha_l \}_{l\in \NN}$ such that  $X(\alpha_l) \rightarrow \bar{X} \in \F$.
\end{lemma}
\begin{proof}
Let $\bar{\alpha}$ and $\{\alpha_k\}_{k\in \NN}$ be as in the hypothesis.  First we show that the sequence $X(\alpha_k)$ is bounded.  For any $k \in \NN$ we have
\[
\lVert X(\alpha_k) \rVert_2 \le \lVert X(\alpha_k)+ (\bar{\alpha} - \alpha_k) I \rVert_2 \le \max_{X\in \F(\bar{\alpha})} \lVert X\rVert_2 < +\infty.
\]
The second inequality is due to $X(\alpha_k) + (\bar{\alpha} - \alpha_k)I \in \F(\bar{\alpha})$ and the third inequality holds since $\F(\bar{\alpha})$ is bounded.
 Thus there exists a convergent subsequence $\{\alpha_l\}_{l\in \NN}$
with $X(\alpha_l) \rightarrow \bar{X}$, that clearly belongs to $\F$.
\end{proof}
For the dual variables we need only prove that $Z(\alpha)$ converges
(for a subseqence) since this implies that $y(\alpha)$ also converges,
by the assumption that $\A$ is surjective.  As for $\Xa$, we show that
the tail of the parametric path corresponding to $Z(\alpha)$ is bounded.
To this end, we first prove the following technical lemma.  Recall that
$\hat{X}$ is the analytic center of Definition~\ref{def:analytic}.
\begin{lemma}
\label{lem:technicalbounded}
Let $\bar{\alpha} > 0$.  There exists $M > 0$ such that for all $ \alpha \in (0,\bar{\alpha}]$,
\[
0 < \langle X(\alpha)^{-1}, \hat{X} + \alpha I \rangle \le M.
\]
\end{lemma}
\begin{proof}
Let $\bar{\alpha}$ be as in the hypothesis and let $\alpha \in (0,\bar{\alpha}]$.  The first inequality is trivial since both of the matrices are positive definite.  For the second inequality, we have,
\begin{equation}
\label{eq:boundednessfirst}
\begin{split}
\langle X(\bar{\alpha})^{-1} - X(\alpha)^{-1}, \hat{X} + \bar{\alpha}I - X(\alpha) \rangle &= \langle \frac{1}{\bar{\alpha}}\A^*(y(\bar{\alpha})) - \frac{1}{\alpha}\A^*(y(\alpha)),   \hat{X} + \bar{\alpha}I - X(\alpha) \rangle, \\
&= \langle \frac{1}{\bar{\alpha}}y(\bar{\alpha}) - \frac{1}{\alpha}y(\alpha),   \A(\hat{X} + \bar{\alpha}I) - \A(X(\alpha)) \rangle, \\
&= \langle \frac{1}{\bar{\alpha}}y(\bar{\alpha}) - \frac{1}{\alpha}y(\alpha),   (\bar{\alpha} - \alpha) \A(I) \rangle, \\
&= \langle X(\bar{\alpha})^{-1} - X(\alpha)^{-1},   (\bar{\alpha} - \alpha) I \rangle, \\
&= (\bar{\alpha} - \alpha)\trace(X(\bar{\alpha})^{-1}) - \langle X(\alpha)^{-1},   (\bar{\alpha} - \alpha) I \rangle.
\end{split}
\end{equation}
On the other hand,
\begin{equation}
\label{eq:boundednesssecond}
\begin{split}
\langle X(\bar{\alpha})^{-1} - X(\alpha)^{-1}, \hat{X} + \bar{\alpha}I - X(\alpha) \rangle &= n + \langle X(\bar{\alpha})^{-1}, \hat{X} \rangle + \bar{\alpha}\trace(X(\bar{\alpha})^{-1}) \\ 
& \qquad \qquad - \langle X(\bar{\alpha})^{-1}, X(\alpha) \rangle - \langle X(\alpha)^{-1}, \hat{X} + \bar{\alpha} I \rangle.
\end{split}
\end{equation}
Combining \eqref{eq:boundednessfirst} and \eqref{eq:boundednesssecond} we get
\begin{align*}
(\bar{\alpha} - \alpha)\trace(X(\bar{\alpha})^{-1}) - \langle X(\alpha)^{-1},   (\bar{\alpha} - \alpha) I \rangle &= n + \langle X(\bar{\alpha})^{-1}, \hat{X} \rangle + \bar{\alpha}\trace(X(\bar{\alpha})^{-1}) \\ 
& \qquad \qquad - \langle X(\bar{\alpha})^{-1}, X(\alpha) \rangle - \langle X(\alpha)^{-1}, \hat{X} + \bar{\alpha} I \rangle.
\end{align*}
After rearranging, we obtain
\begin{equation}
\label{eq:boundednessthird}
\begin{split}
\langle X(\alpha)^{-1}, \hat{X} + \alpha I \rangle &= n + \langle X(\bar{\alpha})^{-1}, \hat{X} \rangle + \bar{\alpha}\trace(X(\bar{\alpha})^{-1})- \langle X(\bar{\alpha})^{-1}, X(\alpha) \rangle \\ 
& \qquad \qquad - (\bar{\alpha} - \alpha)\trace(X(\bar{\alpha})^{-1}), \\
&= n + \alpha \trace(X(\bar{\alpha})^{-1}) + \langle X(\bar{\alpha})^{-1}, \hat{X} \rangle - \langle X(\bar{\alpha})^{-1}, X(\alpha) \rangle.
\end{split}
\end{equation}
The first and the third terms of the right hand side are positive constants.  The second term is positive for every value of $\alpha$ and is bounded above by $\bar{\alpha}\trace(X(\bar{\alpha})^{-1})$ while the fourth term is bounded above by 0.  Applying these bounds as well as the trivial lower bound on the left hand side, we get
\begin{equation}
\label{eq:boundednessfourth}
0 < \langle X(\alpha)^{-1}, \hat{X} + \alpha I \rangle \le n + \bar{\alpha}\trace(X(\bar{\alpha})^{-1})+ \langle X(\bar{\alpha})^{-1}, \hat{X} \rangle =: M.
\end{equation}
\end{proof}
We need one more ingredient to prove that the parametric path
corresponding to $\Za$ is bounded. This involves bounding the trace
inner product above and below by the \textdef{maximal and minimal scalar 
products} of the eigenvalues, respectively.
\begin{lemma}[Ky-Fan \cite{Fan:50}, Hoffman-Wielandt \cite{hw53}]
\label{lem:eigenvaluebound}
If $A,B \in \Sn$, then
\[
\sum_{i=1}^n \lambda_i(A)\lambda_{n+1-i}(B) \le \langle A, B \rangle \le \sum_{i=1}^n \lambda_i(A)\lambda_i(B).
\]
\end{lemma}
We now have the necessary tools for proving boundedness and obtain the
following convergence result.
 \begin{theorem}
 \label{thm:2paramcluster}
Let $\bar{\alpha} >0$.  For every sequence $\{ \alpha_{k} \}_{k \in \NN} \subset (0,\bar{\alpha}]$ such that $\alpha_k \searrow 0$, there exists a subsequence $\{\alpha_{\ell}\}_{\ell \in \NN}$ such that 
 \[
 (X(\alpha_{\ell}),y(\alpha_{\ell}),Z(\alpha_{\ell})) \rightarrow
(\bar{X},\bar{y},\bar{Z}) \in \{\Snp \times \Rm \times \Snp \}
 \]
 with $\bar{X} \in \relint(\F)$ and $\bar{Z} = \A^*(\bar{y})$.
 \end{theorem}
 \begin{proof}
Let $\bar{\alpha} > 0$ and $\{\alpha_k\}_{k\in \NN}$ be as in the hypothesis.  We may without loss of generality assume that $X(\alpha_k) \rightarrow \bar{X} \in \F$ due to Lemma~\ref{lem:primalconverge}.  Let $k \in \NN$.  Combining the upper bound of Lemma~\ref{lem:technicalbounded} with the lower bound of Lemma~\ref{lem:eigenvaluebound} we have 
\[
\sum_{i=1}^n \lambda_i(X(\alpha_{k})^{-1}) \lambda_{n+1-i}(\hat{X}+\alpha_{k} I) \le M.
\]
Since the left hand side is a sum of positive terms, the inequality applies to each term:
\[
\lambda_i(X(\alpha_{k})^{-1}) \lambda_{n+1-i}(\hat{X}+\alpha_{k} I) \le M, \quad \forall i \in \{1,\dotso,n\}.
\]
Equivalently,
\begin{equation}
\label{eq:dualconverge}
\lambda_i(X(\alpha_{k})^{-1}) \le \frac{M}{ \lambda_{n+1-i}(\hat{X}) + \alpha_{k}}, \quad \forall i \in \{1,\dotso,n\}.
\end{equation}
Now exactly $r$ eigenvalues of $\hat{X}$ are positive.  Thus for $i \in \{n-r+1,\dotso,n\}$ we have
\[
\lambda_i(X(\alpha_{k})^{-1}) \le \frac{M}{ \lambda_{n+1-i}(\hat{X}) + \alpha_{k}} \le \frac{M}{ \lambda_{n+1-i}(\hat{X})},
\]
and we conclude that the $r$ smallest eigenvalues of $X(\alpha_{k})^{-1}$ are bounded above.  Consequently, there are at least $r$ eigenvalues of $X(\alpha_{k})$ that are bounded away from 0 and $\rank(\bar{X}) \ge r$.  On the other hand $\bar{X} \in \F$ and $\rank(\bar{X}) \le r$ and it follows that $\bar{X} \in \relint (\F)$.

Now we show that $Z(\alpha_{k})$ is a bounded sequence.  Indeed, from \eqref{eq:dualconverge} we have
\[
\lVert Z(\alpha_{k}) \rVert_2 = \alpha_{k}\lambda_1(X(\alpha_{k})^{-1}) \le \alpha_{k}\frac{M}{ \lambda_n(\hat{X}) + \alpha_{k}} = \alpha_{k}\frac{M}{  \alpha_{k}} = M.
\]
The second to last equality follows from the assumption that $\hat{X} \in \Snp \setminus \Snpp$, i.e. $\lambda_n(\hat{X}) = 0$.  Now there exists a subsequence $\{\alpha_{\ell}\}_{\ell \in \NN}$ such that 
\[
Z(\alpha_{\ell}) \rightarrow \bar{Z}, \ X(\alpha_{\ell}) \rightarrow \bar{X}.
\]
Moreover, for each $\ell$, there exists a unique $y(\alpha_{\ell})\in \Rm$ such that $Z(\alpha_{\ell}) = \A^*(y(\alpha_{\ell}))$ and since $\A$ is surjective, there exists $\bar{y} \in \Rm$ such that $y(\alpha_{\ell}) \rightarrow \bar{y}$ and $\bar{Z} = \A^*(\bar{y})$.  Lastly, the sequence $Z(\alpha_{\ell})$ is contained in the closed cone $\Snp$ hence $\bar{Z} \in \Snp$, completing the proof. 
 \end{proof}
 
We conclude this section by proving that the parametric path is smooth
and has a limit point as $\alpha \searrow 0$.  Our proof relies on the
following lemma of Milnor and is motivated by an analogous proof for the
central path of \SDP in \cite{Halicka:01,HalickaKlerkRoos:01}.  Recall
that an \emph{algebraic set} is the solution set of a system of finitely many polynomial equations.
\begin{lemma}[Milnor \cite{mi68}]
\label{lem:milnor}
Let $\VV \subseteq \Rk$ be an algebraic set and $\UU \subseteq \Rk$ be an open set defined by finitely many polynomial inequalities.  Then if $0 \in \cl (\UU \cap \VV)$ there exists $\varepsilon > 0$ and a real analytic curve $p :[0,\varepsilon) \rightarrow \Rk$ such that $p(0)=0$ and $p(t) \in \UU \cap \VV$ whenever $t > 0$.
\end{lemma}
\begin{theorem}
\label{thm:2paramconverge}
There exists $(\bar{X},\bar{y},\bar{Z}) \in \Snp \times \Rm \times \Snp$ with all the properties of Theorem~\ref{thm:2paramcluster} such that 
\[
\lim_{\alpha \searrow 0} (X(\alpha),y(\alpha),Z(\alpha)) = (\bar{X},\bar{y},\bar{Z}).
\]
\end{theorem}
\begin{proof}
Let $(\bar{X},\bar{y},\bar{Z})$ be a cluster point of the parametric path as in Theorem~\ref{thm:2paramcluster}.  We define the set $\UU$ as 
\[
\UU := \{(X,y,Z, \alpha) \in \Sn \times \Rm \times \Sn \times \R : \bar{X} + X \succ 0, \ \bar{Z} + Z \succ 0, \ Z = \A^*(y), \ \alpha > 0 \}.
\]
Note that each of the positive definite constraints is equivalent to $n$
strict determinant (polynomial) inequalities.  Therefore, $\UU$ satisfies the assumptions of Lemma~\ref{lem:milnor}.  Next, let us define the set $\VV$ as,
\[
\VV := \left \{ (X,y,Z,\alpha) \in \Sn \times \Rm \times \Sn \times \R: \begin{bmatrix}
\A^*(y) - Z  \\ 
\A(X) + \alpha\A(I) \\ 
(\bar{Z}+Z)(\bar{X} + X) - \alpha I 
\end{bmatrix} = 0 \right \},
\]
and note that $\VV$ is indeed a real algebraic set.  Next we show that
there is a one-to-one correspondance between $\UU \cap \VV$ and the
parametric path without any of its cluster points.  Consider
$(\tilde{X},\tilde{y},\tilde{Z},\tilde{\alpha}) \in \UU \cap \VV$ and let
$(X(\tilde{\alpha}),y(\tilde{\alpha}),Z(\tilde{\alpha}))$ be a point on
the parametric path.  We show that
\begin{equation}
\label{eq:2paramfirst}
(\bar{X} + \tilde{X}, \bar{y} + \tilde{y}, \bar{Z} + \tilde{Z}) =
(X(\tilde{\alpha}),y(\tilde{\alpha}),Z(\tilde{\alpha})).
\end{equation}
First of all $\bar{X} + \tilde{X} \succ 0$ and $\bar{Z} + \tilde{Z}
\succ 0$ by inclusion in $\UU$.  Secondly, $(\bar{X} + \tilde{X}, \bar{y}
+ \tilde{y}, \bar{Z} + \tilde{Z})$ solves the system \eqref{eq:scaledoptimality}
when $\alpha = \tilde{\alpha}$:
\[
\begin{bmatrix}
\A^*(\bar{y} + \tilde{y}) - (\bar{Z} + \tilde{Z}) \\
\A(\bar{X} + \tilde{X}) - b(\tilde{\alpha}) \\
(\bar{Z} + \tilde{Z})(\bar{X} + \tilde{X}) - \tilde{\alpha}I
\end{bmatrix} = \begin{bmatrix}
\A^*(\bar{y}) - \bar{Z} + (\A^*(\tilde{y}) - \tilde{Z}) \\
b +\tilde{\alpha}\A(I) - b(\tilde{\alpha}) \\
0
\end{bmatrix} = \begin{bmatrix}
0  \\
0 \\
0
\end{bmatrix}.
\]
Since \eqref{eq:scaledoptimality} has a unique solution, \eqref{eq:2paramfirst} holds.  Thus,
\[
(\tilde{X},\tilde{y},\tilde{Z}) = (X(\alpha) - \bar{X},y(\alpha) - \bar{y}, Z(\alpha)-\bar{Z}),
\]
and it follows that $\UU \cap \VV$ is a translation of the parametric path (without its cluster points):
\begin{equation}
\label{eq:2paramsecond}
\UU \cap \VV = \{(X,y,Z,\alpha) \in \Sn \times \Rm \times \Sn \times \R : (X,y,Z) = (X(\alpha) - \bar{X},y(\alpha) - \bar{y}, Z(\alpha)-\bar{Z}), \ \alpha > 0 \}.
\end{equation}
Next, we show that $0 \in \cl(\UU \cap \VV)$.  To see this, note that 
\[
(X(\alpha),y(\alpha),Z(\alpha)) \rightarrow (\bar{X},\bar{y},\bar{Z}), 
\]
as $\alpha \searrow 0$ along a subsequence.  Therefore, along the same subsequence, we have
\[
( X(\alpha) - \bar{X}, y(\alpha) - \bar{y}, Z(\alpha) - \bar{Z}, \alpha) \rightarrow 0.
\]
Each of the elements of this subsequence belongs to $\UU \cap \VV$ by \eqref{eq:2paramsecond} and therefore $0 \in \cl(\UU \cap \VV)$. 

We have shown that $\UU$ and $\VV$ satisfy all the assumptions of
Lemma~\ref{lem:milnor}, hence there exists $\varepsilon > 0$ and an analytic curve $p:
[0,\varepsilon) \rightarrow \Sn \times \Rm \times \Sn \times \R$ such that $p(0) = 0$ and $p(t) \in \UU \cap \VV$ for $t > 0$.  Let 
\[
p(t) = (X_{(t)},y_{(t)},Z_{(t)},\alpha_{(t)}),
\]
and observe that by \eqref{eq:2paramsecond}, we have
\begin{equation}
\label{eq:2paramthird}
(X_{(t)},y_{(t)},Z_{(t)},\alpha_{(t)}) = (X(\alpha_{(t)}) - \bar{X},y(\alpha_{(t)}) - \bar{y}, Z(\alpha_{(t)})-\bar{Z}).
\end{equation}
Since $p$ is a real analytic curve, the map $g: [0,\varepsilon) \rightarrow \R$ defined as $g(t) = \alpha_{(t)},$ is a differentiable function on the open interval $(0,\varepsilon)$ with 
\[
\lim_{t\searrow 0} g(t) = 0.
\]
In particular, this implies that there is an interval
$[0,\bar{\varepsilon}) \subseteq [0,\varepsilon)$ where $g$ is monotone.
It follows that on $[0,\bar{\varepsilon})$, $g^{-1}$ is a well defined
continuous function that converges to $0$ from the right.  Note that for
any $t > 0$, $(X(t),y(t),Z(t))$ is on the parametric path.  Therefore,
\[
\lim_{t\searrow 0}X(t) = \lim_{t\searrow 0} X(g(g^{-1}(t))) = \lim_{t\searrow 0} X(\alpha_{(g^{-1}(t)}).
\]
Substituting with \eqref{eq:2paramthird}, we have
\[
\lim_{t\searrow 0}X(t) =  \lim_{t\searrow 0} X_{(g^{-1}(t))} + \bar{X} = \bar{X}.
\]
Similarly, $y(t)$ and $Z(t)$ converge to $\bar{y}$ and $\bar{Z}$
respectively.  Thus every cluster point of the parametric path is identical
to $(\bar{X},\bar{y},\bar{Z})$.
\end{proof}
We have shown that the tail of the parametric path is smooth and it has
a limit point.  Smoothness of the entire path follows from Berge's
Maximum Theorem, \cite{MR1464690}, or  \cite[Example 5.22]{MR1491362}.

\subsection{Convergence to the Analytic Center}
\label{sec:analyticcenter}
The results of the previous section establish that the parametric path
converges to $\relint (\F)$ and therefore the primal part of the limit
point has excatly $r$ positive eigenvalues.  If the smallest positive
eigenvalue is very small it may be difficult to distinguish it from zero
numerically. Therefore it is desirable for the limit point to be
`substantially' in the relative interior, in the sense that its smallest
positive eigenvalue is relatively large.  The analytic center has this
property and so a natural question is whether the limit point coincides
with the analytic center.  In the following modification of an example
of \cite{HalickaKlerkRoos:01}, the parametric path converges to a point different from the analytic center.
\begin{example}
\label{ex:noncvg}
Consider the \SDP feasibility problem where $\A$ is defined by
\[S_1 := \begin{bmatrix} 
1 & 0 & 0 & 0 \\
0 & 1 & 0 & 0 \\
0 & 0 & 0 & 0 \\
0 & 0 & 0 & 0 \\
\end{bmatrix},\,\, S_2 :=   
\begin{bmatrix} 
0 & 0 & 0 & 0 \\
0 & 0 & 0 & 1 \\
0 & 0 & 1 & 0 \\
0 & 1 & 0 & 0 \\
\end{bmatrix}, \,\, S_3 :=   \begin{bmatrix} 
0 & 0 & 0 & 0 \\
0 & 0 & 1 & 0 \\
0 & 1 & 0 & 0 \\
0 & 0 & 0 & 1 \\
\end{bmatrix}, \] 
\[S_4 :=   \begin{bmatrix} 
0 & 0 & 0 & 0 \\
0 & 0 & 0 & 0 \\
0 & 0 & 0 & 1 \\
0 & 0 & 1 & 0 \\
\end{bmatrix},\,\, S_5 :=   \begin{bmatrix} 
0 & 0 & 0 & 0 \\
0 & 0 & 0 & 0 \\
0 & 0 & 0 & 0 \\
0 & 0 & 0 & 1 \\
\end{bmatrix},
\]
and $b := (1,0,0,0,0)^T$.  One can verify that the feasible set consists of positive semidefinite matrices of the form
\[X= \begin{bmatrix} 
1-x_{22} & x_{12} & 0 & 0 \\
x_{12} & x_{22} & 0 & 0 \\
0 & 0 & 0 & 0 \\
0 & 0 & 0 & 0 \\
\end{bmatrix}.\]
and the analytic center is the determinant maximizer over the positive definite blocks of this set and satisfies $x_{22}=0.5$ and $x_{12}=0$.  However, the parametric path converges to a matrix with $x_{22} = 0.6$ and $x_{12} = 0$.  To see this note that
\[\A(I) = \begin{pmatrix}2 & 1 & 1 & 0 & 1\end{pmatrix}^T,\quad  b(\alpha) = \begin{pmatrix}1 + 2\alpha & \alpha & \alpha & 0 & \alpha \end{pmatrix}^T.\]
By feasibility, $\Xa$ has the form
\[ \begin{bmatrix} 
1+2\alpha-x_{22} & x_{12} & x_{13} & x_{14} \\
x_{12} & x_{22} & 0 & \frac{1}{2}(\alpha-x_{33}) \\
x_{13} & 0 & x_{33} & 0 \\
x_{14} & \frac{1}{2}(\alpha-x_{33}) & 0 & \alpha \\
\end{bmatrix}.\]
Moreover, the optimality conditions of Theorem~\ref{thm:maxdet} indicate that $\Xoa^{-1} \in \range(\A^*)$ and hence is of the form
\[ \begin{bmatrix} 
* & 0 & 0 & 0 \\
0 & * & * & * \\
0& * & * & * \\
0 & * & * & * \\
\end{bmatrix}.\]
It follows that $x_{12}=x_{13}=x_{14} = 0$ and $\Xa$ has the form
\[ 
\begin{bmatrix} 
1 + 2\alpha -x_{22} & 0 & 0 & 0 \\
0 & x_{22} & 0 & \frac{1}{2}(\alpha-x_{33}) \\
0 & 0 & x_{33} & 0 \\
0 & \frac{1}{2}(\alpha-x_{33}) & 0 & \alpha \\
\end{bmatrix}. 
\] 
Of all the matrices with this form, $\Xoa$ is the one maximizing the determinant, that is
\begin{align*}
([\Xoa]_{22}, [\Xoa]_{33})^T = \arg \max \ & x_{33}(1+2\alpha - x_{22})(\alpha x_{22} - \frac{1}{4}(\alpha - x_{33})^2), \\
s.t. \ &  0 < x_{22} < 1+2\alpha, \\
& x_{33} > 0, \\
& \alpha x_{22} >  \frac{1}{4}(\alpha - x_{33})^2.
\end{align*}
Due to the strict inequalities, the maximizer is a stationary point of the objective function.  Computing the derivative with respect to $x_{22}$ and $x_{33}$ we obtain the equations
\begin{align*}
x_{33}(-(\alpha x_{22} - \frac{1}{4}(\alpha - x_{33})^2) + \alpha(1+2\alpha-x_{22}) &= 0, \\
(1+2\alpha-x_{22})((\alpha x_{22} - \frac{1}{4}(\alpha - x_{33})^2) + \frac{1}{2}x_{33}(\alpha-x_{33})) &= 0.
\end{align*}
Since $x_{33} > 0$ and $(1+2\alpha-x_{22}) > 0$, we may divide them out.  Then solving each equation for $x_{22}$ we get
\begin{align}
\label{ex:first}
x_{22} &= \frac{1}{8\alpha}(\alpha - x_{33})^2 + \alpha + \frac{1}{2}, \\
\label{ex:second}
x_{22} &= \frac{1}{4\alpha}(\alpha - x_{33})^2 - \frac{1}{2\alpha}x_{33}(\alpha-x_{33}).
\end{align}
Substituting \eqref{ex:first} into \eqref{ex:second} we get
\begin{align*}
0 &= \frac{1}{4\alpha}(\alpha - x_{33})^2 - \frac{1}{2\alpha}x_{33}(\alpha-x_{33}) - \frac{1}{8\alpha}(\alpha - x_{33})^2 - \alpha - \frac{1}{2}, \\
&= \frac{1}{8\alpha}(\alpha - x_{33})^2 - \frac{1}{2}x_{33} + \frac{1}{2\alpha}x_{33}^2 - \alpha - \frac{1}{2}, \\
&= \frac{1}{8\alpha}x_{33}^2 - \frac{1}{4}x_{33} +\frac{1}{8}\alpha - \frac{1}{2}x_{33} + \frac{1}{2\alpha}x_{33}^2 - \alpha - \frac{1}{2}, \\
&= \frac{5}{8\alpha}x_{33}^2 - \frac{3}{4}x_{33} +\frac{1}{8}\alpha - \alpha - \frac{1}{2}, \\
\end{align*}
Now we solve for $x_{33}$,
\begin{align*}
x_{33} &= \frac{\frac{3}{4} \pm \sqrt{ \frac{9}{16} - 4(\frac{5}{8\alpha})(\frac{1}{8}\alpha - \alpha - \frac{1}{2})}}{2\frac{5}{8\alpha}}, \\
&= \frac{3\alpha}{5} \pm \frac{4\alpha}{5}\sqrt{ \frac{11\alpha + 5}{4\alpha}}, \\
&= \frac{1}{5}(3\alpha + 2\sqrt{\alpha}\sqrt{ 11\alpha + 5}).
\end{align*}
Since $x_{33}$ is fully determined by the stationarity constraints, we have $[\Xoa]_{33} = x_{33}$ and $[\Xa]_{33} \rightarrow 0$ as $\alpha \searrow 0$.  Substituting this expression for $x_{33}$ into \eqref{ex:first} we get
\begin{align*}
[\Xoa]_{22} &= \frac{1}{8\alpha}(\alpha - \frac{1}{5}(3\alpha + 2\sqrt{\alpha}\sqrt{ 11\alpha + 5}))^2 + \alpha + \frac{1}{2}, \\
&= \frac{1}{8\alpha}(\alpha^2 - 2\alpha \frac{1}{5}(3\alpha + 2\sqrt{\alpha}\sqrt{ 11\alpha + 5}) + \frac{1}{25}(9\alpha^2 + 6\alpha \sqrt{\alpha}\sqrt{ 11\alpha + 5} + 4\alpha(11\alpha+5))) + \alpha + \frac{1}{2}, \\
&= \frac{1}{8}\alpha - \frac{1}{20}(3\alpha + 2\sqrt{\alpha}\sqrt{ 11\alpha + 5}) + \frac{1}{200}(9\alpha + 6 \sqrt{\alpha}\sqrt{ 11\alpha + 5} + 4(11\alpha+5)) + \alpha + \frac{1}{2}, \\
&= \frac{31}{25}\alpha - \frac{7}{100}\sqrt{\alpha}\sqrt{ 11\alpha + 5} + \frac{6}{10}.
\end{align*}
Now it is clear that $[\Xoa]_{22} \rightarrow 0.6$ as $\alpha \searrow 0$.
\end{example}

\subsubsection{A Sufficient Condition for Convergence to the Analytic Center}
\label{sec:sufficientanalytic}
Recall that $\face(\F) = V\Srp V^T$.  To simplify the discussion we may assume that $V = \begin{bmatrix} I \\ 0 \end{bmatrix}$, so that
\begin{equation}
\label{eq:facialstructure}
\face(\F) = \begin{bmatrix}
\Srp &0 \\
0 & 0
\end{bmatrix}.
 \end{equation}
This follows from the rich automorphism group of $\Snp$, that is, for any full rank $W\in \Rnn$, we have $W\Snp W^T = \Snp$.  Moreover, it is easy to see that there is a one-to-one correspondence between relative interior points under such transformations.

Let us now express $\F$ in terms of $\nul(\A)$, that is, if $A_0 \in \F$ and 
recall that $A_1,\dotso, A_q, \, q=t(n)-m,$ form a basis for $\nul (\A)$, then
\index{$\nul(\A)=\spanl \{ A_1,\dotso,A_q\}$}
\[
\F = \left( A_0 + \spanl \{ A_1,\dotso,A_q\} \right) \cap \Snp.
\]
Similarly,
\[
\Fa = \left( \alpha I + A_0 + \spanl \{ A_1,\dotso,A_q\} \right) \cap \Snp.
\]
Next, let us partition $A_i$ according to the block structure of \eqref{eq:facialstructure}:
\begin{equation}
\label{eq:partNi}
A_i = \begin{bmatrix} L_i & M_i \cr M_i^T & N_i \end{bmatrix} , \quad i\in \{0, \ldots , q\}.
\end{equation}
Since $A_0 \in \F$, from \eqref{eq:facialstructure} we have $N_0 = 0$ and $M_0 = 0$.  Much of the subsequent discussion focuses on the linear pencil $\sum_{i=1}^q x_iN_i$.  Let $\N$ be the linear mapping such that 
\[
\nul (\N) = \left \{ \sum_{i=1}^q x_iN_i : x \in \Rq \right \}.
\]
\begin{lemma}
\label{lem:maxdetN} 
Let $\{N_1,\dotso,N_q\}$ be as in \eqref{eq:partNi},
$\spanl \{N_1,\dotso,N_q\} \cap \Snp = \{0\}$, and let
\begin{equation}
\label{eq:Q}
Q := \arg \max \{\log \det (X): X = I + \sum_{i=1}^q x_i N_i \succ 0, \ x \in \Rq \}.
\end{equation}
Then for all $\alpha >0$,
\begin{equation}
\label{eq:alphaQ}
\alpha Q = \arg \max \{\log \det (X): X = \alpha I + \sum_{i=1}^q x_i N_i \succ 0, \ x \in \Rq \}.
\end{equation} 
\end{lemma}

\begin{proof} 
We begin by expressing $Q$ in terms of $\N$:
\[
Q = \arg \max \{\log \det (X): \N(X) = \N(I) \}.
\]
By the assumption on the span of the matrices $N_i$ and by Lemma~\ref{lem:boundedchar}, the feasible set of \eqref{eq:Q} is bounded.  Moreover, the feasible set contains positive definite matrices, hence all the assumptions of Theorem~\ref{thm:maxdet} are satisfied.  It follows that $Q$ is the unique feasible, positive definite matrix satisfying $Q^{-1} \in \range( \N^*)$.  

Moreover, $\alpha Q$ is positive definite, feasible for \eqref{eq:alphaQ}, 
and $(\alpha Q)^{-1} \in \range(\N^*)$.  Therefore $\alpha Q$ is optimal 
for \eqref{eq:alphaQ}.
\end{proof}

Now we prove that the parametric path converges to the analytic center under the condition of Lemma~\ref{lem:maxdetN}.
\begin{theorem}
\label{thm:analyticcenter} 
Let $\{N_1,\dotso,N_q\}$ be as in \eqref{eq:partNi}.
If $\spanl \{N_1,\dotso,N_q\} \cap \Snp = \{0\}$ and $\bar{X}$ is the limit point of the primal part of the parametric path as in Theorem~\ref{thm:2paramconverge}, then $\bar{X} = \hat{X}$.
\end{theorem}
\begin{proof} 
Let 
\[
\bar{X} =: \begin{bmatrix}
\bar{Y} & 0 \\
0 & 0
\end{bmatrix},\ \hat{X} =: \begin{bmatrix}
\hat{Y} & 0 \\
0 & 0
\end{bmatrix}
\]
and suppose, for eventual contradiction, that $\bar{Y} \ne \hat{Y}$.  Then let $r,s \in \R$ be such that 
\[
\det(\bar{Y}) < r < s < \det(\hat{Y}).
\]
Let $Q$ be as in Lemma~\ref{lem:maxdetN} and let $x \in \Rq$ satisfy $Q = I + \sum_{i=1}^q x_iN_i$.  Now for any $\alpha >0$ we have
\[
\hat{X} + \alpha( I + \sum_{i=1}^q x_i A_i) =  \begin{pmatrix} \hat{Y}+\alpha I + \alpha \sum_{i=1}^q x_i L_i & \alpha \sum_{i=1}^q x_iM_i \cr \alpha \sum_{i=1}^q x_i M_i^T & \alpha Q \end{pmatrix} .
\] 
Note that there exists $\varepsilon >0$ such that $\hat{X} + \alpha \sum_{i=1}^q x_iA_i \succeq 0$ whenever $\alpha \in (0,\varepsilon)$.  It follows that 
\[
\hat{X} + \alpha( I + \sum_{i=1}^q x_i A_i) \in \Fa, \quad \forall \alpha \in (0,\varepsilon).
\]
Taking the determinant, we have
\begin{align*}
\frac{1}{\alpha^{n-r}} \det (\hat{X} + \alpha( I + \sum_{i=1}^q x_i A_i)) &= \frac{1}{\alpha^{n-r}}\det \left( \alpha Q-\alpha^2 (\sum_{i=1}^q x_iM_i ) (\hat{Y}+\alpha I + \alpha \sum_{i=1}^q x_iL_i )^{-1} (\sum_{i=1}^q x_iM_i^T) \right) \\
&\qquad \qquad \times\det (\hat{Y}+\alpha I + \alpha \sum_{i=1}^q x_iL_i), \\
&=  \det \left( Q-\alpha (\sum_{i=1}^q x_iM_i ) (\hat{Y}+\alpha I + \alpha \sum_{i=1}^q x_iL_i )^{-1} (\sum_{i=1}^q x_iM_i^T) \right) \\
&\qquad \qquad \times\det (\hat{Y}+\alpha I + \alpha \sum_{i=1}^q x_iL_i).
\end{align*}
Now we have
\[
\lim_{\alpha \searrow 0} \  \frac{1}{\alpha^{n-r}} \det (\hat{X} + \alpha( I + \sum_{i=1}^q x_i A_i)) = \det(Q)\det(\hat{Y}).
\]
Thus, there exists $\sigma \in (0,\varepsilon)$ so that for $\alpha \in (0,\sigma )$ we have
\[
\det (\hat{X} + \alpha( I + \sum_{i=1}^q x_i A_i)) > s \alpha^{n-r} \det (Q) .
\]
As $\Xa$ is the determinant maximizer over $\Fa$, we also have
\begin{equation}
\label{eq:detX} 
\det( \Xa) > s \alpha^{n-r} \det( Q ), \quad \forall \alpha \in (0, \sigma ).  
\end{equation}
On the other hand $\Xa \rightarrow \bar{X}$ and let
\[
\Xa =: \begin{bmatrix}
\alpha I + \sum_{i=1}^q x(\alpha)_i L_i & \sum_{i=1}^q x(\alpha)_i M_i \\
\sum_{i=1}^q x(\alpha)_i M^T_i & \alpha I + \sum_{i=1}^q x(\alpha)_i N_i
\end{bmatrix}.
\]
Then $\alpha I + \sum_{i=1}^q x(\alpha)_i L_i \rightarrow \bar{Y}$ and there exists $\delta \in (0,\sigma)$ such that for all $\alpha \in (0,\delta)$,
\[
\det(\alpha I + \sum_{i=1}^q x(\alpha)_i L_i) < r.
\]
Moreover, by definition of $Q$,
\[
\det(\alpha I + \sum_{i=1}^q x(\alpha)_i N_i) \le \det(\alpha Q) = \alpha^{n-r} \det(Q).
\]
To complete the proof, we apply the Hadamard-Fischer inequality to $\det(\Xa)$.  For $\alpha \in (0,\delta)$ we have
\[
\det(\Xa) \le \det(\alpha I + \sum_{i=1}^q x(\alpha)_i L_i)\det(\alpha I + \sum_{i=1}^q x(\alpha)_i N_i) < r\alpha^{n-r} \det( Q),
\]
a contradiction of \eqref{eq:detX}.
\end{proof}
\begin{remark}
Note that Example \ref{ex:noncvg} fails the hypotheses of Theorem
\vspace{.1in}
\ref{thm:analyticcenter}. Indeed, the matrix
$\begin{bmatrix} 
0 & 0 & 0 & 0 \\
0 & 0 & 0 & -1 \\
0 & 0 & 2 & 0 \\
0 & -1 & 0 & 0 \\
\end{bmatrix}$
lies in $\nul(\A)$ and the bottom $2\times 2$ block is nonzero and positive
semidefinite.
\end{remark}

\section{The Projected Gauss-Newton Method}
\label{sec:projGN}

We have constructed a parametric path that converges to a point in the
relative interior of $\F$.  In this section we propose an algorithm to
follow the path to its limit point.  We do not prove convergence of the
proposed algorithm and address its performance in
Section~\ref{sec:numerics}. We follow the (projected) Gauss-Newton approach (the
nonlinear analog of Newton's method) originally introduced for \SDPs in
\cite{KrMuReVaWo:98} and improved more recently in \cite{KrukDoanW:10}.
This approach has been shown to have improved robustness compared to
other symmetrization approaches. For well posed problems,
the Jacobian for the search direction remains full rank in the limit to
the optimum.

\subsection{Scaled Optimality Conditions}
The idea behind this approach is to view the system defining the
parametric path as an overdetermined
map and use the Gauss-Newton (GN) method for nonlinear systems.
In the process, the linear feasibility equations are eliminated and the
GN method is applied to the remaining bilinear equation.
For $\alpha \ge0$ let $G_{\alpha}: \Snp \times \Rm \times \Snp \rightarrow
\Sn \times \Rm \times \Rnn$ be defined as
\begin{equation}
\label{eq:GdefGN}
G_{\alpha}(X,y,Z):= 
\begin{bmatrix}
    \A^*(y)-Z \\
    \A(X) -\ba \\
   ZX - \alpha I \\
\end{bmatrix}. 
\end{equation}
The solution to $G_{\alpha}(X,y,Z)= 0$ is exactly $(\Xa,\ya,\Za)$ when
$\alpha > 0$; and for $\alpha = 0$ the solution set is
\[
\F \times (\A^*)^{-1}(\DD) \times \DD , \quad \DD := \range(\A^*) \cap \face(\F)^c.
\]
Clearly, the limit point of the parametric path satisfies $G_0(X,y,Z) = 0$.  We fix $\alpha >
0$. The GN direction, $(dX,dy,dZ)$, uses the
overdetermined \textdef{GN system}
\index{$(dX,dy,dZ)$, Gauss-Newton direction} 
\index{Gauss-Newton direction, $(dX,dy,dZ)$} 
\begin{equation}
\label{eq:GNorig}
G_{\alpha}'(X,y,Z)\begin{bmatrix}
dX \\
dy \\
dZ
\end{bmatrix} = -G_{\alpha}(X,y,Z).
\end{equation}
Note that the search direction is a strict descent direction for the norm of
the residual, $\| \kvec(G_{\alpha}(X,y,Z)) \|_2^2$, when the Jacobian is full rank.
The size of the problem is then reduced by projecting out the first two equations. We are left with a single linearization of the bilinear
complementarity equation, i.e.,~$n^2$ equations in only $t(n)$ variables.
The \textdef{least squares solution} yields the projected
GN direction after backsolves.
We prefer steps of length $1$, however, the primal and dual step lengths, $\alpha_p$ and $\alpha_d$ respectively,
are reduced, when necessary, 
to ensure strict feasibility: $X + \alpha_p dX \succ 0$ and 
$Z+\alpha_d dZ \succ 0$.
The parameter $\alpha$ is then reduced and the 
procedure repeated.  On the parametric path, $\alpha$ satisfies
\begin{equation}
\label{eq:alpharep}
\alpha = \frac{\langle \Za, \Xa \rangle }{n}.
\end{equation}
Therefore, this is a good estimate of the target for $\alpha$ near the
parametric path. As is customary, we then use a fixed $\sigma \in (0,1)$ to
move the target towards optimality, $\alpha \leftarrow \sigma \alpha$. 

\subsubsection{Linearization and GN Search Direction}
For the purposes of this discussion we vectorize the variables and data
 in $G_{\alpha}$. Let $A \in \R^{m\times t(n)}$ be the matrix representation 
of $\A$, that is
\index{$A$, matrix representation}
\index{matrix representation, $A$}
\[
A_{i,:} := \svec(S_i)^T, \quad i\in \{1,\dotso,m\}.
\]
Let $N \in \R^{t(n)\times (t(n)-m)}$ be such that its columns form a
basis for $\nul (A)$ and let $\hat{x}$ be a particular solution to
$Ax=\ba$, e.g., the least squares solution.  Then the affine manifold
determined from the equation $\A(X)=\ba$ is equivalent to that
obtained from the equation
\[
x = \hat{x} + Nv, \quad v\in \R^{t(n)-m}. 
\]
Moreover, if $z:=\svec(Z)$, we have the vectorization
\begin{equation}
g_{\alpha}(x,v,y,z) := \begin{bmatrix}
A^Ty - z \\
x-\hat{x}-Nv \\
\sMat(z)\sMat(x) - \alpha I
\end{bmatrix} =: \begin{bmatrix}
r_d \\
r_p \\
R_c
\end{bmatrix},
\label{eq:systemg}
\end{equation}
Now we show how the first two equations of the above system may be
projected out, thereby reducing the size of the problem.  First we have
\[
g'_{\alpha}(x,v,y,z)\begin{pmatrix}
dx \\
dv \\ 
dy \\ 
dz
\end{pmatrix} = \begin{bmatrix}
A^Tdy - dz \\
dx - Ndv \\
\sMat(dz)\sMat(x) + \sMat(z)sMat(dx)
\end{bmatrix},
\]
and it follows that the GN step as in \eqref{eq:GNorig} is the least squares solution of the system
\[
\begin{bmatrix}
A^Tdy - dz \\
dx - Ndv \\
\sMat(dz)\sMat(x) + \sMat(z)\sMat(dx)
\end{bmatrix} = - \begin{bmatrix}
r_d \\
r_p \\
R_c \\
\end{bmatrix}.
\]
Since the first two equations are linear, we get $dz = A^Tdy+r_d$ and $dx = Ndv - r_p$.  Substituting into the third equation we have,
\[
\sMat(A^Tdy + r_d)\sMat(x) + \sMat(z)\sMat(Ndv - r_p) = -R_c.
\]
After moving all the constants to the right hand side we obtain the projected GN system in $dy$ and $dv$,
\begin{equation}
\label{eq:projGN}
\sMat(A^Tdy)\sMat(x) + \sMat(z)\sMat(Ndv) = -R_c + \sMat(z)\sMat(r_p) - \sMat(r_d)\sMat(x).
\end{equation}
The least squares solution to this system is the exact GN direction when $r_d = 0$ and $r_p=0$, otherwise it is an approximation.  We then use the equations $dz = A^Tdy+r_d$ and $dx = Ndv - r_p$ to obtain search directions for $x$ and $z$.

In  \cite[Theorem 1]{KrukDoanW:10}, it is proved that if the solution
set of $G_0(X,y,Z) = 0$ is a singleton such that $X+Z \succ 0$ and the
starting point of the projected GN algorithm is sufficiently
close to the parametric path then the algorithm, with a crossover
modification, converges quadratically.  
As we showed
above, the solution set to our problem is 
\[
\F \times (\A^*)^{-1}(\DD) \times \DD,
\]
which is not a singleton as long as $\F \ne \emptyset$.  Indeed, $\DD$ is a non-empty cone.  Although the convergence result of \cite{KrukDoanW:10} does not apply to our problem, their numerical tests indicate that the algorithm converges even for problems violating the strict complementarity and uniqueness assumptions and our observations agree.

\subsection{Implementation Details}
Several specific implementation modifications are used.  We begin with initial $x,v,y,z$ with corresponding $X,Z\succ 0$. If we obtain $P \succ 0$ as in Proposition~\ref{prop:boundtest} then we set $Z = P$ and define $y$ accordingly, otherwise $Z = X = I$.  We
estimate $\alpha$ using \eqref{eq:alpharep} and set $\alpha \leftarrow
2\alpha$ to ensure that our target is somewhat well centered to start.

\subsubsection{Step Lengths and Linear Feasibility}
We start with initial step lengths $\alpha_p=\alpha_d=1.1$ and then 
backtrack using a Cholesky factorization test to ensure positive definiteness 
\[
X+\alpha_p dX \succ 0, \quad Z+\alpha_d dZ \succ 0.
\]
If the step length we find is still $>1$ after the backtrack, 
we set it to $1$ and first update $v,y$ and then update $x,z$ using
\[
x=\hat x + N v,  \quad z=A^Ty.
\]
This ensures exact linear feasibility. Thus we find that we maintain exact dual
feasibility after a few iterations. Primal feasibility changes since
$\alpha$ decreases. We have experimented with including an extra few iterations at the end of the algorithm
with a fixed $\alpha$ to obtain exact primal feasibility (for the given $\alpha$).  In most cases the improvement of feasibility with respect to $\F$ was minimal and not worth the extra computational cost.

\subsubsection{Updating $\alpha$ and Expected Number of Iterations}
In order to drive $\alpha$ down to zero, we fix $\sigma \in (0,1)$ and
update alpha as $\alpha \leftarrow \sigma \alpha$.  We use a moderate $\sigma =
.6$.  However, if this reduction is performed too
quickly then our step lengths end up being too small and we get too close
to the positive semidefinite boundary. Therefore, we change $\alpha$
using information from $\min \{\alpha_p,\alpha_d\}$. If the steplength
is reasonably near $1$ then we decrease using $\sigma$; if the steplength is
around $.5$ then we leave $\alpha$ as is; if the steplength is small
then we \emph{increase} to $1.2\alpha$; and if the steplength is tiny
($<.1$), we increase to $2\alpha$. For most of the test problems,
this strategy resulted in steplengths of $1$ after the first few
iterations.

We noted empirically that the condition number of the Jacobian for the
least squares problem increases quickly, i.e.,~several singular values
converge to zero. Despite this we are able to
obtain high accuracy search directions.\footnote{Our algorithm finds
the search direction using \eqref{eq:projGN}. If we looked at a singular
value decomposition then we get the equivalent system $\Sigma
(V^T d \bar s) = (U^T RHS)$. We observed that several singular
values in $\Sigma$ converge to zero while the corresponding elements
in $(U^T RHS)$ converge to zero at a similar rate. This accounts for the
improved accuracy despite the huge condition numbers. This appears to be a
similar phenomenon to that observed in the analysis of interior point
methods in \cite{MR99i:90093,MR96f:65055} and as discussed in
\cite{GoWo:04}.}

Since we typically have steplengths of $1$, $\alpha$ is generally decreased using $\sigma$. Therefore, for a desired tolerance 
$\epsilon$ and a starting $\alpha =1$ we would want $\sigma^k < \epsilon$, or equivalently,
\[
\quad k < \log_{10} (\epsilon)/\log_{10}(\sigma).  
\]
For our $\sigma=.6$ and $t$ decimals of desired accuracy, we expect to need
$k<4.5t$ iterations.

\section{Generating Instances and Numerical Results}
\label{sec:numerics}
In this section we analyze the performance of an implementation of our algorithm.  We begin with a discussion on generating spectrahedra.  A particular challenge is in creating spectrahedra with specified singularity degree.  Following this discussion, we present and analyze the numerical results.

\subsection{Generating Instances with Varying Singularity Degree}
\label{sec:generating}
Our method for generating instances is motivated by the approach of \cite{WeiWolk:06} for generating \SDPs with varying \emph{complementarity gaps}.  We begin by proving a relationship between strict complementarity of a primal-dual pair of \SDP problems and the singularity degree of the optimal set of the primal \SDP.  This relationship allows us to modify the code presented in \cite{WeiWolk:06} and obtain spectrahedra having various singularity degrees.  Recall the primal \SDP 
\begin{equation}
\label{prob:sdpprimalcopy}
\SDP \qquad \qquad \textdef{$p^{\star}$}:=\min \{ \langle C,X\rangle : \A(X)=b, X\succeq 0\},
\end{equation}
with dual
\begin{equation}
\label{prob:sdpdualcopy}
\DSDP \qquad \qquad \textdef{$d^{\star}$}:=\min \{ b^Ty  : \A^*(y) \preceq C \}.
\end{equation}
Let $O_P\subseteq \Snp$ and $O_D\subseteq \Snp$ denote the primal and dual optimal sets respectively, where the dual optimal set is with respect to the variable $Z$.  Specifically,
\[
O_P := \{X\in \Snp : \A(X) = b, \ \langle C, X\rangle = p^{\star} \}, \ O_D := \{ Z \in \Snp : Z = C-\A^*(y), \ b^Ty = d^{\star}, \ y \in \Rm  \}.
\]
Note that $O_P$ is a spectrahedron determined by the affine manifold
\[
\begin{bmatrix}
\A(X) \\
\langle C, X \rangle
\end{bmatrix} = \begin{pmatrix}
b \\
p^*
\end{pmatrix}.
\]
We note that the second system in the theorem of the alternative, Theorem~\ref{thm:alternative}, for the spectrahedron $O_P$ is
\begin{equation}
\label{eq:opalternative}
0 \ne \tau C + \A^*(y) \succeq 0, \ \tau p^{\star} + y^Tb = 0.
\end{equation}
We say that \emph{strict complementarity} holds for \SDP and \DSDP if there exists $X^{\star} \in O_P$ and $Z^{\star}\in O_D$ such that
\[
\langle X^{\star}, Z^{\star} \rangle = 0  \text{ and } \rank(X^{\star}) + \rank(Z^{\star}) = n.
\]
If strict complementarity does not hold for \SDP and \DSDP and there exist $X^{\star} \in \relint(O_P)$ and $Z^{\star} \in \relint(O_D)$, then we define the complementarity gap as
\[
g := n - \rank(X^{\star}) - \rank(Z^{\star}).
\]
Now we describe the relationship between strict complementarity of \SDP and \DSDP and the singularity degree of $O_P$.
\begin{prop}
\label{prop:scsd}
If strict complementarity holds for
\SDP and \DSDP, then $\sd(O_P) \le 1$.
\end{prop}
\begin{proof}
Let $X^{\star} \in \relint(O_P)$.  If $X^{\star} \succ 0$, then $\sd(O_P) = 0$ and we are done.  Thus we may assume $\rank(X^{\star}) < n$.  By strict complementarity, there exists $(y^{\star},Z^{\star}) \in \Rm \times \Snp$ feasible for \DSDP with $Z^{\star} \in O_D$ and $\rank(X^{\star}) + \rank(Z^{\star}) = n$.  Now we show that $(1,-y^{\star})$ satisfies \eqref{eq:opalternative}.  Indeed, by dual feasibility,
\[
C - \A^*(y^{\star}) = Z^{\star} \in \Snp \setminus \{0\},
\]
and by complementary slackness,
\[
p^{\star} - (y^{\star})^Tb = \langle X^{\star}, C \rangle - \langle \A^*(y^{\star}), X^{\star} \rangle = \langle X^{\star},Z^{\star} \rangle = 0.
\]
Finally, since $\rank(X^{\star}) + \rank(Z^{\star}) = n$ we have $\sd(O_P) = 1$, as desired.
\end{proof}

From the perspective of facial reduction, the interesting spectrahedra are those with singularity degree greater than zero and the above proposition gives us a way to construct spectrahedra with singularity degree exactly one.  Using the algorithm of \cite{WeiWolk:06} we construct strictly complementary \SDPs and then use the optimal set of the primal to construct a spectrahedron with singularity degree exactly one.  Specifically, given positive integers $n, m, r,$ and $g$ the algorithm of \cite{WeiWolk:06} returns the data $\A,b,C$ corresponding to a primal dual pair of \SDPs, together with $X^{\star} \in \relint(O_P)$ and $Z^{\star} \in \relint(O_D)$ satisfying
\[
\rank(X^{\star}) = r, \ \rank(Z^{\star}) = n-r-g.
\]
Now if we set 
\[
\hat{\A}(X) := \begin{pmatrix}
\A(X) \\
\langle C, X \rangle
\end{pmatrix}, \ \hat{b} = \begin{pmatrix}
b \\
\langle C, X^{\star} \rangle
\end{pmatrix},
\]
then $O_P = \F(\hat{\A},\hat{b})$.  Moreover, if $g=0$ then $\sd(O_P) = 1$, by Proposition~\ref{prop:scsd}.
This approach could also be used to create spectrahedra with larger singularity degrees by constructing \SDPs with greater complementarity gaps, if the converse of Proposition~\ref{prop:scsd} were true.  We provide a sufficient condition for the converse in the following proposition.
\begin{prop}
\label{prop:sdscconverse}
If $\sd(O_P)=0$, then strict complementarity holds for \SDP and \DSDP.  Moreover, if $\sd(O_P) = 1$ and the set of solutions to \eqref{eq:opalternative} intersects $\R_{++} \times \Rm$, then strict complementarity holds for \SDP and \DSDP.  
\end{prop}
\begin{proof}
Since we have only defined singularity degree for non-empty spectrahedra, there exists $X^{\star} \in \relint(O_P)$.  For the first statement, by Theorem~\ref{thm:strongduality}, there exists $Z^{\star} \in O_D$.  Complementary slackness always holds, hence $\langle Z^{\star}, X^{\star} \rangle = 0$ and since $X^{\star}\succ 0$ we have $Z^{\star}=0$.  It follows that $\rank(X^{\star}) + \rank(Z^{\star}) = n$ and strict complementarity holds for \SDP and \DSDP.  

For the second statement, let $(\bar{\tau},\bar{y})$ and $(\tilde{\tau},\tilde{y})$ be solutions to \eqref{eq:opalternative} with $\bar{\tau} >0$ and $\tilde{\tau} C +\A^*(\tilde{y})$ of maximal rank.  Let 
\[
\bar{Z} := \bar{\tau}C + \A^*(\bar{\tau}), \ \tilde{Z} := \tilde{\tau} C +\A^*(\tilde{y}).
\]
Then there exists $\varepsilon > 0$ such that $\bar{\tau} + \varepsilon \tilde{\tau} >0$ and $\rank(\bar{Z} + \varepsilon \tilde{Z}) \ge \rank(\tilde{Z})$.  Define
\[
\tau := \bar{\tau} + \varepsilon \tilde{\tau}, \ y := \bar{y} + \varepsilon \tilde{y}, \ Z :=  \bar{Z} + \varepsilon \tilde{Z}.
\]
Now $(\tau,y)$ is a solution to \eqref{eq:opalternative}, i.e.,
\[
0 \ne \tau C + \A^*(y) \succeq 0, \ \tau p^{\star} + y^Tb = 0.
\] 
Moreover, $\rank(X^{\star}) + \rank(Z) = n$ since $\sd(O_P) = 1$ and $Z$ is of maximal rank.  Now we define
\[
Z^{\star} := \frac{1}{\tau} Z = C - \A^*\left(-\frac{1}{\tau}y\right).
\]
Since $\tau>0$, it is clear that $Z^{\star} \succeq 0$ and it follows that $\left(-\frac{1}{\tau} y, Z^{\star}\right)$ is feasible for \DSDP.  Moreover, this point is optimal since
\[
d^{\star} \ge -\frac{1}{\tau} y^Tb = p^{\star}\ge d^{\star}.
\]
Therefore $Z^{\star} \in O_D$ and since $\rank(Z^{\star}) = \rank(Z)$, strict complementarity holds for \SDP and \DSDP.
\end{proof}

\subsection{Numerical Results}
\label{sec:numericsreal}

\begin{table}[ht]
\centering
\begin{tabular}{|c|c|c|c|c|c|c|c|c|c|} \hline
$n$ & $m$ & $r$ & $\lambda_1(X)$ & $\lambda_r(X)$ & $\lambda_{r+1}(X)$ & $\lambda_n(X)$ & $\lVert \A(X)-b \rVert_2$ & $\langle Z,X\rangle$ & $\alpha_f$ \cr\hline
   50 &   100 &    25 & 1.06e+02 & 2.80e+01 & 1.97e-11 & 5.07e-13 & 3.17e-12 & 1.26e-13 & 1.10e-12 \cr\hline
   80 &   160 &    40 & 8.74e+01 & 3.22e+01 & 1.20e-10 & 9.00e-13 & 7.28e-12 & 2.95e-13 & 2.01e-12 \cr\hline
  110 &   220 &    55 & 7.74e+01 & 3.73e+01 & 3.56e-10 & 7.23e-13 & 9.12e-12 & 3.65e-13 & 2.14e-12 \cr\hline
  140 &   280 &    70 & 7.82e+01 & 3.84e+01 & 4.11e-10 & 7.08e-13 & 1.26e-11 & 5.20e-13 & 2.65e-12 \cr\hline
\end{tabular}

\caption{Results for the case $\sd=1$.  The eigenvalues refer to those of the primal variable, $X$, and each entry is the average of five runs.}
\label{tab:sd1}
\end{table}

\begin{table}[h!]
\centering
\begin{tabular}{|c|c|c|c|c|c|c|} \hline
$n$ & $m$ & $r$ & $\lambda_1(Z)$ & $\lambda_{r_d}(Z)$ & $\lambda_{r_d+1}(Z)$ & $\lambda_n(Z)$ \cr\hline
   50 &   100 &    25 & 1.85e+00 & 9.07e-02 & 3.96e-14 & 1.27e-14 \cr\hline
   80 &   160 &    40 & 1.96e+00 & 6.91e-02 & 6.23e-14 & 2.30e-14 \cr\hline
  110 &   220 &    55 & 1.98e+00 & 2.61e-02 & 5.77e-14 & 2.78e-14 \cr\hline
  140 &   280 &    70 & 2.03e+00 & 2.46e-02 & 6.96e-14 & 3.39e-14 \cr\hline
\end{tabular}

\caption{Eigenvalues of the dual variable, $Z$, corresponding to the primal variable of Table~\ref{tab:sd1}.  Each entry is the average of five runs.}
\label{tab:sd1dual}
\end{table}

\begin{table}[h!]
\centering
\tiny{
\begin{tabular}{|c|c|c|c|c|c|c|c|c|c|c|c|c|} \hline
$n$ & $m$ & $r$ & $g$ & $\lambda_1(X)$ & $\lambda_r(X)$ & $\lambda_{r+1}(X)$ & $\lambda_{r+g}(X)$ & $\lambda_{r+g+1}(X)$ & $\lambda_n(X)$ & $\lVert \A(X)-b \rVert_2$ & $\langle Z,X\rangle$ & $\alpha_f$ \cr\hline
   50 &   100 &    17 &     5 & 9.89e+01 & 1.85e+01 & 6.62e-05 & 2.61e-05 & 2.13e-10 & 6.10e-13 & 4.99e-12 & 2.04e-13 & 1.07e-12 \cr\hline
   80 &   160 &    27 &     8 & 1.11e+02 & 2.00e+01 & 1.89e-05 & 1.28e-05 & 7.36e-11 & 5.17e-13 & 8.40e-12 & 2.73e-13 & 1.27e-12 \cr\hline
  110 &   220 &    37 &    11 & 1.09e+02 & 2.42e+01 & 3.52e-05 & 2.33e-05 & 2.05e-10 & 1.52e-12 & 1.92e-11 & 6.46e-13 & 2.33e-12 \cr\hline
  140 &   280 &    47 &    14 & 1.63e+02 & 2.64e+01 & 1.07e-04 & 2.65e-05 & 1.02e-10 & 1.17e-13 & 9.84e-12 & 3.52e-13 & 1.48e-12 \cr\hline
\end{tabular}

}
\caption{Results for the case $\sd=2$.  The eigenvalues refer to those of the primal variable, $X$, and each entry is the average of five runs.}
\label{tab:sd2}
\end{table}

\begin{table}[h!]
\centering
\begin{tabular}{|c|c|c|c|c|c|c|c|c|c|} \hline
$n$ & $m$ & $r$ & $g$ & $\lambda_1(Z)$ & $\lambda_{r_d}(Z)$ & $\lambda_{r_d+1}(Z)$ & $\lambda_{r_d+g}(Z)$ & $\lambda_{r_d+g+1}(Z)$ & $\lambda_n(Z)$ \cr\hline
   50 &   100 &    17 &     5 & 2.22e+00 & 2.51e-02 & 1.04e-07 & 8.38e-08 & 9.18e-14 & 1.51e-14 \cr\hline
   80 &   160 &    27 &     8 & 2.03e+00 & 3.65e-02 & 1.03e-07 & 7.45e-08 & 7.92e-14 & 1.69e-14 \cr\hline
  110 &   220 &    37 &    11 & 2.13e+00 & 6.11e-02 & 1.78e-07 & 1.23e-07 & 1.36e-13 & 2.76e-14 \cr\hline
  140 &   280 &    47 &    14 & 2.19e+00 & 4.16e-02 & 7.39e-08 & 4.35e-08 & 6.04e-14 & 8.14e-15 \cr\hline
\end{tabular}

\caption{Eigenvalues of the dual variable, $Z$, corresponding to the primal variable of Table~\ref{tab:sd2}.  Each entry is the average of five runs.}
\label{tab:sd2dual}
\end{table}

For the numerical tests, we generate instances with $n \in \{50,80,110,140\}$ and $m=2n$.  These are problems of small size relative to state of the art capabilities, nonetheless, we are able to demonstrate the performance of our algorithm through them.  In Table~\ref{tab:sd1} and Table~\ref{tab:sd1dual} we record the results for the case $\sd=1$.  For each instance, specified by $n$, $m,$ and $r$, the results are the average of five runs.  By $r$, we denote the maximum rank over all elements of the generated spectrahedron, which is fixed to $r=n/2$.  In Table~\ref{tab:sd1} we record the relevant eigenvalues of the primal variable, primal feasibility, complementarity, and the value of $\alpha$ at termination, denoted $\alpha_f$.  The values for primal feasibility and complementarity are sufficiently small and it is clear from the eigenvalues presented, that the first $r$ eigenvalues are significantly smaller than the last $n-r$.  These results demonstrate that the algorithm returns a matrix which is very close to the relative interior of $\F$.  In Table~\ref{tab:sd1dual} we record the relevant eigenvalues for the corresponding dual variable, $Z$.  Note that $r_d := n-r$ and the eigenvalues recorded in the table indicate that $Z$ is indeed an exposing vector.  Moreover, it is a maximal rank exposing vector.  While, we have not proved this, we observed that it is true for every test we ran with $\sd = 1$.  

In Table~\ref{tab:sd2} and Table~\ref{tab:sd2dual} we record similar values for problems where the singularity degree may be greater than $1$.  Using the approach described in Section~\ref{sec:generating} we generate instances of \SDP and \DSDP having a complementarity gap of $g$ and then we construct our spectrahedron from the optimal set of \SDP.  By Proposition~\ref{prop:scsd} and Proposition~\ref{prop:sdscconverse} the resulting spectrahedron may have singularity degree greater than 1.  We observe that primal feasibility and complementarity are attained to a similar accuracy as in the $\sd=1$ case.  The eigenvalues of the primal variable fall into three categories.  The first $r$ eigenvalues are sufficiently large so as not to be confused with $0$, the last $n-r-g$ eigenvalues are convincingly small, and the third group of eigenvalues, exactly $g$ of them, are such that it is difficult to decide if they should be $0$ or not.  A similar phenomenon is observed for the eigenvalues of the dual variable.  This demonstrates that exactly $g$ of the eigenvalues are converging to $0$ at a rate significantly smaller than that of the other $n-r-g$ eigenvalues.

\section{An Application to PSD Completions of Simple Cycles}
\label{sec:psdcyclecompl}
In this final section, we show that our parametric path and the relative interior point it converges to have interesting structure for cycle completion problems.

Let $G=(V,E)$ be an undirected graph with $n = \lvert V \rvert$ and let
$a \in \R^{\lvert E \rvert}$.  Let us index the components of $a$ by the
elements of $E$.  A matrix $X \in \Sn$ is a \textdef{completion} of $G$
under $a$ if $X_{ij} = a_{ij}$ for all $\{i,j\} \in E$.  We say that $G$
is \textdef{partially PSD} under $a$ if there exists a completion of $G$
under $a$ such that all of its principle minors consisting entirely of
$a_{ij}$ are PSD.  Finally, we say that $G$ is \textdef{PSD completable}
if for all $a$ such that $G$ is partially PSD, there exists a PSD
completion.  Recall that a graph is \textdef{chordal} if for every cycle with at
least four vertices, there is an edge connecting non-adjacent vertices.
The classical result of \cite{GrJoSaWo:84} states that $G$ is PSD
completable if, and only if, it is chordal.

An interesting problem for non-chordal graphs is to characterize the
vectors $a$ for which $G$ admits a PSD completion.  Here we consider PSD
completions of non-chordal cycles with loops.  This problem was first
looked at in \cite{MR1236734}, where the following special case is presented.

\begin{theorem}[Corollary 6, \cite{MR1236734}]
\label{thm:simplecycle}
Let $n\ge 4$ and $\theta, \phi \in [0,\pi]$. Then
\begin{equation}\label{simple} C := \begin{bmatrix}
1 & \cos(\theta) & & & \cos(\phi) \\
\cos(\theta) & 1 & \cos(\theta) & ? & \\
& \cos(\theta) & 1 & \ddots & \\
& ? & \ddots & \ddots & \cos(\theta) \\
\cos(\phi) & & & \cos(\theta) & 1
\end{bmatrix}, \end{equation}
has a positive semidefinite completion if, and only if,
$$ \phi \le (n-1)\theta \le (n-2)\pi + \phi \qquad \text{for n even}$$
and
$$ \phi \le (n-1)\theta \le (n-1)\pi - \phi \qquad \text{for n odd.}$$
The partial matrix \eqref{simple} has a positive definite completion if,
and only if, the above inequalities are strict. 
\end{theorem}

Using the results of the previous sections we present an analytic expression for exposing vectors in the case where a PSD completion exists but not a PD one, i.e., the Slater CQ does not hold for the corresponding \SDP.  We begin by showing that the primal part of the parametric path is always Toeplitz.  In general, for a partial Toeplitz matrix, the unique maximum determinant completion is not necessarily Toeplitz. For instance, the maximum determinant completion of
$$ \begin{bmatrix} 6 & 1 & x & 1 & 1 \cr 1 & 6 & 1 & y & 1 \cr x & 1 & 6 & 1 & z \cr 1 & y & 1 & 6 & 1 \cr 1 & 1 & z & 1 & 6 \end{bmatrix} $$
is given by $x=z=0.3113$ and $y=0.4247$.

\begin{theorem}\label{md} 
If the parital matrix
\[
P := \begin{bmatrix}
a & b & & & c \\
b & a & b & ? & \\
& b & a & \ddots & \\
& ? & \ddots & \ddots & b \\
c & & & b & a
\end{bmatrix}
\]
has a positive definite completion, then the unique maximum determinant completion is Toeplitz.
\end{theorem}

First we present the following technical lemma.  Let $J_n \in \Sn$ be the matrix with ones on the antidiagonal and zeros everywhere else, that is, $[J_n]_{ij} = 1$ when $i+j=n+1$ and zero otherwise. For instance,  $J_2=\begin{bmatrix} 0 & 1 \cr 1 & 0 \end{bmatrix}$.

\begin{lemma}\label{persymm} If $A$ is the maximum determinant completion of $P$, then $A=JAJ$. 
\end{lemma}

\begin{proof} As $A$ is a completion of $P$, so is $JAJ$. Furthermore, $\det (A) = \det (JAJ)$. Since the maximum determinant completion is unique, we must have that $A=JAJ$. \end{proof}

\begin{proof}[Proof of Theorem~\ref{md}]
The proof is by induction on the size $n$. When $n=4$ the result follows from Lemma \ref{persymm}.

Suppose Theorem~\ref{md} holds for size $n-1$. Let $A$ be the maximum determinant completion of $P$. 
Then by the optimality conditions of Theorem~\ref{thm:maxdet},
$$A^{-1}= \begin{bmatrix}
* & * &0 & \cdots  & 0 & * \\
* & * & * & 0 & \ddots & 0\\
0 & * & * & * & \ddots &  \vdots \\
\vdots & \ddots & \ddots & \ddots & \ddots & 0\\
0 & \cdots & 0 & * & * & * \\
* & 0 & \cdots & 0& *& *
\end{bmatrix}. $$
Let $\alpha := A_{1,n-1}$, and consider the $(n-1)\times (n-1)$ partial matrix 
\begin{equation}\label{simple2} \begin{bmatrix}
a & b & & & \alpha \\
b & a & b & ? & \\
& b & a & \ddots & \\
& ? & \ddots & \ddots & b \\
\alpha & & & b & a
\end{bmatrix}, \end{equation}
By the induction assumption, \eqref{simple2} has a Toeplitz maximum determinant completion, say $B$. 
Note that 
\begin{equation}\label{simple4} B^{-1}= \begin{bmatrix}
* & * &0 & \cdots  & 0 & * \\
* & * & * & 0 & \ddots & 0\\
0 & * & * & * & \ddots &  \vdots \\
\vdots & \ddots & \ddots & \ddots & \ddots & 0\\
0 & \cdots & 0 & * & * & * \\
* & 0 & \cdots & 0& *& *
\end{bmatrix}. \end{equation}
Now consider the partial matrix
\begin{equation}\label{simple3} \begin{bmatrix} B & \begin{bmatrix} c \cr ? \cr \vdots \cr ? \cr b \end{bmatrix} \cr 
 \begin{bmatrix} c & ? & \cdots & ? & b \end{bmatrix} & a \end{bmatrix} 
\end{equation}
Since this is a chordal pattern we only need to check that the fully prescribed principal minors are positive definite. These are $B$ and 
$$ \begin{bmatrix} a & \alpha & c \cr \alpha & a & b \cr c & b & a
\end{bmatrix} , $$ the latter of which is a principal submatrix of the
positive definite matrix $A$. Thus \eqref{simple3} has a maximum determinant completion, say $C$. Then 
$$C^{-1} = \begin{bmatrix} * &  \begin{bmatrix} * \cr 0 \cr \vdots \cr 0 \cr * \end{bmatrix} \cr 
 \begin{bmatrix} * & 0 & \cdots & 0 & * \end{bmatrix} & * \end{bmatrix} = : \begin{bmatrix} L & M \cr M^T & N \end{bmatrix} . $$  
By the properties of block inversion,
$$ C = \begin{bmatrix} (L-MN^{-1}M^T)^{-1} & * \cr * & * \end{bmatrix} = \begin{bmatrix} B & * \cr * & * \end{bmatrix} , $$
and it follows that $B^{-1} = L-MN^{-1}M^T$. Since $MN^{-1}M^T$ only has nonzero entries in the four corners, we obtain that 
$$L=\begin{bmatrix}
* & * &0 & \cdots  & 0 & * \\
* & * & * & 0 & \ddots & 0\\
0 & * & * & * & \ddots &  \vdots \\
\vdots & \ddots & \ddots & \ddots & \ddots & 0\\
0 & \cdots & 0 & * & * & * \\
* & 0 & \cdots & 0& *& *
\end{bmatrix}.
$$
We now see that $C^{-1}$ and $A^{-1}$ have zeros in all entries $(i,j)$ with $|i-j| >1$ and $(i,j) \not\in\{(1,n-1), (1,n), (n-1,1) , (n,1)\}$. Also, $A$ and $C$ have the same entries in positions $(i,j)$ where $|i-j|\le 1$ or where $(i,j) \in\{(1,n-1), (1,n), (n-1,1) , (n,1)\}$. But then $A$ and $C$ are two positive definite matrices where for each $(i,j)$ either $A_{ij}=C_{ij}$ or $(A^{-1})_{ij} = (C^{-1})_{ij}$, yielding that $A=C$ (see, e.g., \cite{MR1321785}). Finally, observe that the Toeplitz matrix $B$ is the $(n-1)\times (n-1)$ upper left submatrix of $C$, and that $A=JAJ$, to conclude that $A$ is Toeplitz.
\end{proof}
When \eqref{simple} has a PD completion, then this result states that the analytic center of all the completions is Toeplitz.  When \eqref{simple} has a PSD completion, but not a PD completion then the primal part of the parametric path is always Toeplitz and since the Toeplitz matrices are closed, \eqref{simple} admits a maximum rank Toeplitz PSD completion.  In the following proposition we see that the dual part of the parametric path has a specific form.
\begin{prop} 
\label{Tinverse}
Let $T=(t_{i-j})_{i,j=1}^n$ be a positive definite real Toeplitz matrix, and suppose that $(T^{-1})_{k,1}=0$ for all $k\in \{3,\ldots , n-1\}$. Then $T^{-1}$ has the form 
$$\begin{bmatrix}
a & c & 0& & d \\
c & b & c & \ddots & \\
0& c & b & \ddots &0 \\
& \ddots & \ddots & \ddots & c \\
d & & 0& c & a
\end{bmatrix},$$
with $b=\frac{1}{a} (a^2+c^2-d^2)$.
\end{prop}

\begin{proof} 
Let us denote the first column of $T$ by $\begin{bmatrix} a & c & 0 &
\cdots & 0 & d \end{bmatrix}^T$. By the \textdef{Gohberg-Semencul
formula} (see \cite{MR0353038,MR1038316}) we have that
$$ T^{-1} =\frac{1}{a} ( AA^T-BB^T ), $$
where 
$$ A=\begin{bmatrix}
a & 0 & 0& & 0 \\
c & a & 0 & \ddots & \\
0& c & a & \ddots &0 \\
& \ddots & \ddots & \ddots & 0 \\
d & & 0& c & a
\end{bmatrix}, B= \begin{bmatrix}
0 & 0 & 0& & 0 \\
d & 0 & 0 & \ddots & \\
0& d & 0 & \ddots &0 \\
& \ddots & \ddots & \ddots & 0 \\
c & & 0& d & 0
\end{bmatrix}.$$
\end{proof}

\begin{corollary} 
\label{cor:expvecsimplecycle}
If the set of PSD completions of \eqref{simple} is contained in a proper face of $\Snp$ then there exists an exposing vector of the form
$$C_E := \begin{bmatrix}
a & c & 0& & d \\
c & b & c & \ddots & \\
0& c & b & \ddots &0 \\
& \ddots & \ddots & \ddots & c \\
d & & 0& c & a
\end{bmatrix},$$
for a face containing the completions.  Moreover, $C_E$ satisfies
$$2\cos(\theta)c + b = 0 \quad \text{and} \quad a + \cos(\theta)c + \cos(\phi)d = 0.$$
\end{corollary}
\begin{proof}
Existence follows from Proposition~\ref{Tinverse}.  By definition, $C_E$
is an exposing vector for the face if, and only if, $C_E \succeq 0$ and
$\langle X, C_E\rangle = 0$ for all positive semidefinite completions, $X$, of
$C$.  Since $X$ and $C_E$ are positive semidefinite, we have $XC_E = 0$ and in particular $\diag(XC_E) = 0$, which
is satisfied if, and only if,
$$\cos(\theta)c + b + \cos(\theta)c = 0 \quad \text{and} \quad a + \cos(\theta)c + \cos(\phi)d = 0,$$
as desired.
\end{proof}

\section{Conclusion}
\label{sec:conclusion}
In this paper we have considered a `primal' approach to facial reduction for \SDPs that reduces to finding a relative interior point of a spectrahedron.  By considering a parametric optimization problem, we constructed a smooth path and proved that its limit point is in the relative interior of the spectrahedron.  Moreover, we gave a sufficient condition for the relative interior point to coincide with the analytic center.  We proposed a projected Gauss-Newton algorithm to follow the parametric path to the limit point and in the numerical results we observed that the algorithm converges.  We also presented a method for constructing spectrahedra with singularity degree $1$ and provided a sufficient condition for constructing spectrahedra of larger singularity degree.  Finally, we showed that the parametric path has interesting structure for the simple cycle completion problem.  

This research has also highlighted some new problems to be pursued.  We single out two such problems.  The first regards the eigenvalues of the limit point that are neither sufficiently small to be deemed zero nor sufficiently large to be considered as non-zero.  We have experimented with some eigenvalue deflation techniques, but none have led to a satisfactory method.  Secondly, there does not seem to be a method in the literature for constructing spectrahedra with specified singularity degree.

\printindex
\addcontentsline{toc}{section}{Index}
\label{ind:index}

\bibliography{.master,.edm,.psd,.bjorBOOK,.qap}

\def\cprime{$'$} \def\cprime{$'$} \def\cprime{$'$}
  \def\udot#1{\ifmmode\oalign{$#1$\crcr\hidewidth.\hidewidth
  }\else\oalign{#1\crcr\hidewidth.\hidewidth}\fi} \def\cprime{$'$}
  \def\cprime{$'$} \def\cprime{$'$} \def\cprime{$'$}
\begin{thebibliography}{10}

\bibitem{MR1321785}
M.~Bakonyi and H.J. Woerdeman.
\newblock Maximum entropy elements in the intersection of an affine space and
  the cone of positive definite matrices.
\newblock {\em SIAM J. Matrix Anal. Appl.}, 16(2):369--376, 1995.

\bibitem{MR2807419}
M.~Bakonyi and H.J. Woerdeman.
\newblock {\em Matrix completions, moments, and sums of {H}ermitian squares}.
\newblock Princeton University Press, Princeton, NJ, 2011.

\bibitem{MR1236734}
Wayne Barrett, Charles~R. Johnson, and Pablo Tarazaga.
\newblock The real positive definite completion problem for a simple cycle.
\newblock {\em Linear Algebra Appl.}, 192:3--31, 1993.
\newblock Computational linear algebra in algebraic and related problems
  (Essen, 1992).

\bibitem{MR1464690}
C.~Berge.
\newblock {\em Topological spaces}.
\newblock Dover Publications, Inc., Mineola, NY, 1997.
\newblock Including a treatment of multi-valued functions, vector spaces and
  convexity, Translated from the French original by E. M. Patterson, Reprint of
  the 1963 translation.

\bibitem{bw2}
J.M. Borwein and H.~Wolkowicz.
\newblock Characterization of optimality for the abstract convex program with
  finite-dimensional range.
\newblock {\em J. Austral. Math. Soc. Ser. A}, 30(4):390--411, 1980/81.

\bibitem{bw1}
J.M. Borwein and H.~Wolkowicz.
\newblock Facial reduction for a cone-convex programming problem.
\newblock {\em J. Austral. Math. Soc. Ser. A}, 30(3):369--380, 1980/81.

\bibitem{bw3}
J.M. Borwein and H.~Wolkowicz.
\newblock Regularizing the abstract convex program.
\newblock {\em J. Math. Anal. Appl.}, 83(2):495--530, 1981.

\bibitem{ScTuWonumeric:07}
Y-L. Cheung, S.~Schurr, and H.~Wolkowicz.
\newblock Preprocessing and regularization for degenerate semidefinite
  programs.
\newblock In D.H. Bailey, H.H. Bauschke, P.~Borwein, F.~Garvan, M.~Thera,
  J.~Vanderwerff, and H.~Wolkowicz, editors, {\em Computational and
  {A}nalytical {M}athematics, {I}n {H}onor of {J}onathan {B}orwein's 60th
  {B}irthday}, volume~50 of {\em Springer Proceedings in Mathematics \&
  Statistics}, pages 225--276. Springer, 2013.

\bibitem{int:deklerk7}
E.~de~Klerk, C.~Roos, and T.~Terlaky.
\newblock Infeasible--start semidefinite programming algorithms via self--dual
  embeddings.
\newblock In {\em Topics in Semidefinite and Interior-Point Methods}, volume~18
  of {\em The Fields Institute for Research in Mathematical Sciences,
  Communications Series}, pages 215--236. American Mathematical Society, 1998.

\bibitem{KrukDoanW:10}
X.V. Doan, S.~Kruk, and H.~Wolkowicz.
\newblock A robust algorithm for semidefinite programming.
\newblock {\em Optim. Methods Softw.}, 27(4-5):667--693, 2012.

\bibitem{DrusWolk:16}
D.~Drusvyatskiy and H.~Wolkowicz.
\newblock The many faces of degeneracy in conic optimization.
\newblock {\em Foundations and Trends in Optimization}, 2017.
\newblock submitted, 97 pages.

\bibitem{MR3622250}
Mirjam D\"ur, Bolor Jargalsaikhan, and Georg Still.
\newblock Genericity results in linear conic programming---a tour d'horizon.
\newblock {\em Math. Oper. Res.}, 42(1):77--94, 2017.

\bibitem{Fan:50}
K.~FAN.
\newblock On a theorem of weyl concerning eigenvalues of linear transformations
  ii.
\newblock {\em Proc.\ Nat.\ Acad.\ Sci.\ U.S.A.}, 36:31--35, 1950.

\bibitem{fazelhindiboyd:01}
M.~Fazel, H.~Hindi, and S.P. Boyd.
\newblock A rank minimization heuristic with application to minimum order
  system approximation.
\newblock In {\em Proceedings American Control Conference}, pages 4734--4739,
  2001.

\bibitem{MR0353038}
I.C. Gohberg and A.A. Semencul.
\newblock The inversion of finite {T}oeplitz matrices and their continual
  analogues.
\newblock {\em Mat. Issled.}, 7(2(24)):201--223, 290, 1972.

\bibitem{GoWo:04}
M.~Gonzalez-Lima, H.~Wei, and H.~Wolkowicz.
\newblock A stable primal-dual approach for linear programming under
  nondegeneracy assumptions.
\newblock {\em Comput. Optim. Appl.}, 44(2):213--247, 2009.

\bibitem{GrJoSaWo:84}
B.~Grone, C.R. Johnson, E.~Marques~de Sa, and H.~Wolkowicz.
\newblock Positive definite completions of partial {H}ermitian matrices.
\newblock {\em Linear Algebra Appl.}, 58:109--124, 1984.

\bibitem{Halicka:01}
M.~Halick{\'a}.
\newblock Analyticity of the central path at the boundary point in semidefinite
  programming.
\newblock {\em European J. Oper. Res.}, 143(2):311--324, 2002.
\newblock Interior point methods (Budapest, 2000).

\bibitem{HalickaKlerkRoos:01}
M.~Halick{\'a}, E.~de~Klerk, and C.~Roos.
\newblock On the convergence of the central path in semidefinite optimization.
\newblock {\em SIAM J. Optim.}, 12(4):1090--1099 (electronic), 2002.

\bibitem{hw53}
A.J. Hoffman and H.W. Wielandt.
\newblock The variation of the spectrum of a normal matrix.
\newblock {\em Duke Mathematics}, 20:37--39, 1953.

\bibitem{MR1038316}
T.~Kailath and J.~Chun.
\newblock Generalized {G}ohberg-{S}emencul formulas for matrix inversion.
\newblock In {\em The {G}ohberg anniversary collection, {V}ol.\ {I} ({C}algary,
  {AB}, 1988)}, volume~40 of {\em Oper. Theory Adv. Appl.}, pages 231--246.
  Birkh\"auser, Basel, 1989.

\bibitem{KrMuReVaWo:98}
S.~Kruk, M.~Muramatsu, F.~Rendl, R.J. Vanderbei, and H.~Wolkowicz.
\newblock The {G}auss-{N}ewton direction in semidefinite programming.
\newblock {\em Optim. Methods Softw.}, 15(1):1--28, 2001.

\bibitem{LuoStZh:97}
Z-Q. Luo, J.F. Sturm, and S.~Zhang.
\newblock Duality results for conic convex programming.
\newblock Technical Report Report 9719/A, April, Erasmus University Rotterdam,
  Econometric Institute EUR, P.O. Box 1738, 3000 DR, The Netherlands, 1997.

\bibitem{lusz00}
Z-Q. Luo, J.F. Sturm, and S.~Zhang.
\newblock Conic convex programming and self-dual embedding.
\newblock {\em Optim. Methods Softw.}, 14(3):169--218, 2000.

\bibitem{mi68}
J.~Milnor.
\newblock {\em Singular points of complex hypersurfaces}.
\newblock Annals of Mathematics Studies, No. 61. Princeton University Press,
  Princeton, N.J.; University of Tokyo Press, Tokyo, 1968.

\bibitem{MR3108446}
G.~Pataki.
\newblock Strong duality in conic linear programming: facial reduction and
  extended duals.
\newblock In David Bailey, Heinz~H. Bauschke, Frank Garvan, Michel Thera,
  Jon~D. Vanderwerff, and Henry Wolkowicz, editors, {\em Computational and
  analytical mathematics}, volume~50 of {\em Springer Proc. Math. Stat.}, pages
  613--634. Springer, New York, 2013.

\bibitem{Pataki2017}
G.~Pataki.
\newblock Bad semidefinite programs: They all look the same.
\newblock {\em SIAM J. Optim.}, 27(1):146--172, 2017.

\bibitem{permfribergandersen}
F.~Permenter, H.~Friberg, and E.~Andersen.
\newblock Solving conic optimization problems via self-dual embedding and
  facial reduction: a unified approach.
\newblock Technical report, MIT, Boston, MA, 2015.

\bibitem{perm}
F.~Permenter and P.~Parrilo.
\newblock Partial facial reduction: simplified, equivalent {SDP}s via
  approximations of the {PSD} cone.
\newblock Technical Report Preprint arXiv:1408.4685, MIT, Boston, MA, 2014.

\bibitem{2016arXiv160802090P}
F.~{Permenter} and P.~A. {Parrilo}.
\newblock {Dimension reduction for semidefinite programs via Jordan algebras}.
\newblock {\em ArXiv e-prints}, August 2016.

\bibitem{Ram:95}
M.V. Ramana.
\newblock An exact duality theory for semidefinite programming and its
  complexity implications.
\newblock {\em Math. Programming}, 77(2):129--162, 1997.

\bibitem{RaTuWo:95}
M.V. Ramana, L.~Tun{\c{c}}el, and H.~Wolkowicz.
\newblock Strong duality for semidefinite programming.
\newblock {\em SIAM J. Optim.}, 7(3):641--662, 1997.

\bibitem{con:70}
R.T. Rockafellar.
\newblock {\em Convex analysis}.
\newblock Princeton Mathematical Series, No. 28. Princeton University Press,
  Princeton, N.J., 1970.

\bibitem{MR1491362}
R.T. Rockafellar and R.J.-B. Wets.
\newblock {\em Variational analysis}, volume 317 of {\em Grundlehren der
  Mathematischen Wissenschaften [Fundamental Principles of Mathematical
  Sciences]}.
\newblock Springer-Verlag, Berlin, 1998.

\bibitem{S98lmi}
J.F. Sturm.
\newblock Error bounds for linear matrix inequalities.
\newblock {\em SIAM J. Optim.}, 10(4):1228--1248 (electronic), 2000.

\bibitem{MR2724357}
L.~Tun{\c{c}}el.
\newblock {\em Polyhedral and Semidefinite Programming Methods in Combinatorial
  Optimization}, volume~27 of {\em Fields Institute Monographs}.
\newblock American Mathematical Society, Providence, RI, 2010.

\bibitem{MR1614078}
L.~Vandenberghe, S.~Boyd, and S-P. Wu.
\newblock Determinant maximization with linear matrix inequality constraints.
\newblock {\em SIAM J. Matrix Anal. Appl.}, 19(2):499--533, 1998.

\bibitem{waki_mur_sparse}
H.~Waki and M.~Muramatsu.
\newblock A facial reduction algorithm for finding sparse {SOS}
  representations.
\newblock {\em Oper. Res. Lett.}, 38(5):361--365, 2010.

\bibitem{MR3063940}
H.~Waki and M.~Muramatsu.
\newblock Facial reduction algorithms for conic optimization problems.
\newblock {\em J. Optim. Theory Appl.}, 158(1):188--215, 2013.

\bibitem{WeiWolk:06}
H.~Wei and H.~Wolkowicz.
\newblock Generating and measuring instances of hard semidefinite programs.
\newblock {\em Math. Program.}, 125(1, Ser. A):31--45, 2010.

\bibitem{SaVaWo:97}
H.~Wolkowicz, R.~Saigal, and L.~Vandenberghe, editors.
\newblock {\em Handbook of semidefinite programming}.
\newblock International Series in Operations Research \& Management Science,
  27. Kluwer Academic Publishers, Boston, MA, 2000.
\newblock Theory, algorithms, and applications.

\bibitem{MR99i:90093}
M.H. Wright.
\newblock Ill-conditioning and computational error in interior methods for
  nonlinear programming.
\newblock {\em SIAM J. Optim.}, 9(1):84--111 (electronic), 1999.

\bibitem{MR96f:65055}
S.J. Wright.
\newblock Stability of linear equations solvers in interior-point methods.
\newblock {\em SIAM J. Matrix Anal. Appl.}, 16(4):1287--1307, 1995.

\end{thebibliography}
\addcontentsline{toc}{section}{Bibliography}

\end{document}